\documentclass[11pt,a4paper,reqno]{amsart}
\usepackage{a4wide}
\usepackage{amsfonts,amsmath,amssymb,amsthm,enumerate,color,bbm,tabularx,ifthen,twoopt}
\usepackage{graphicx}
\usepackage[latin1]{inputenc}
\usepackage[left]{lineno}
\usepackage[draft]{fixme}
\def\rme{\mathrm{e}}
\def\rmi{\mathrm{i}}

\newenvironment{enum_W}
  {%
  \setlength{\leftmargini}{4em}\begin{enumerate}}
  {\end{enumerate}}

  {\end{enumerate}}

  {\end{enumerate}}

\def\med{\mathrm{med}}
\def\MAD{\mathrm{MAD}}

\def\rme{\mathrm{e}}

\def\rme{\mathrm{e}}
\def\rmi{\mathrm{i}}
\def\ltwo{\mathrm{L}^2}

\def\cl{\stackrel{\mathcal{L}}{\longrightarrow}}

\def\dwt{W}
\def\bdwt{\mathbf{\dwt}}
\def\indexset{\mathcal{I}}

\def\L{\mathrm{T}}

\def\be{\mathbf{e}}

\def\diffop{\mathbf{\Delta}}

\def\1{\mathbbm{1}}

\def\bx{\mathbf{x}}

\def\Nset{\mathbb{N}} 
\def\Zset{\mathbb{Z}}

\def\Rset{\mathbb{R}} 
\def\PE{\mathbb{E}} 
\def\PVar{\mathrm{Var}}
\def\PCov{\mathrm{Cov}}

\def\calN{\mathcal{N}}

\def\calH{\mathcal{H}}

\def\ie{\textit{i.e.} }

\newcommand{\eqdef}{\ensuremath{\stackrel{\mathrm{def}}{=}}}
\newcommand{\eqsp}{\;}

\def\1{\mathbbm{1}}
\newcommand{\AVvar}[3][]
{
\ifthenelse{\equal{#1}{}}{\mathbf{V}_{#3}(#2)}{\mathbf{V}_{#3}(#2,#1)}}
\newcommand{\AVvarJoint}[3][]
{
\ifthenelse{\equal{#1}{}}{\mathbf{W}_{#3}(#2)}{\mathbf{W}_{#3}(#2,#1)}}

\newcommand{\AsympVarWWE}[2][]
{\mathrm{V}(#2)}
\newcommand{\AVvarInv}[3][]
{\ifthenelse{\equal{#1}{}}{\mathbf{V}^{-1}_{#3}(#2)}{\mathbf{V}^{-1}_{#3}(#2,#1)}}
\newcommand{\sigmaasymp}[2][]
{
\ifthenelse{\equal{#1}{}}{\sigma(#2)}{\sigma(#2,#1)}}

\def\vjsymb{\sigma}
\newcommand{\vj}[4][]{%
\ifthenelse{\equal{#1}{}}{\vjsymb^{#4}_{#2}(#3)}{\vjsymb^{#4}_{#2}(#3,#1)}}
\newcommand{\stdj}[3][]{%
\ifthenelse{\equal{#1}{}}{\vjsymb_{#2}(#3)}{\vjsymb_{#2}(#3,#1)}}

\def\rmd{\mathrm{d}}

\newcommand{\hvj}[3][]{%
\ifthenelse{\equal{#1}{}}{\hat{\vjsymb}^2_{#2}}{\hat{\vjsymb}^2_{#2}(#1)}}
\newcommand{\Kvar}[2][]{
\mathrm{K}(#2)}

\def\densletter{\mathbf{D}}
\newcommand{\bdens}[4][]{%
\densletter_{#2}({#3};#4)}

\def\densasympletter{D}

\newcommand{\bdensasymp}[4][]{%
\mathbf{\densasympletter}_{\infty,#2}({#3};#4)}
\newcommand{\bdensasympconv}[5][]{\mathbf{\densasympletter}^{#2}_{\infty,#3}({#4};#5)}
\newcommand{\bdenssingle}[6][]{\mathbf{\densasympletter}^{#2}_{#3,#4}({#5};#6)}

\def\regressweights{\mathbf{w}}

\def\d{\mathrm{d}}

\def\F{\textbf{F}}

\def\allWA{\ref{item:Wreg}--\ref{item:Wvstd}}

\renewcommand{\hat}{\widehat}
\renewcommand{\tilde}{\widetilde}

\newcommand\pscal[2]{\langle #1, #2 \rangle}

\def\X{\underline{\textbf{X}}}

\newcommand\pent[1]{\left\lfloor #1 \right\rfloor}

\def \2{2^{J_2-j}}
\def \2p{2^{J_2-j'}}

\newcounter{hypA}

\def\cd{\stackrel{d}{\longrightarrow}}

\newcommandtwoopt{\QRC}[4][][]
{\ifthenelse{\equal{#1}{}}
{\ifthenelse{\equal{#2}{}}{\mathrm{Q}_{#3}\left( #4\right)}{\mathrm{Q}_{#3}\left(#4,#2\right)}}
{\ifthenelse{\equal{#2}{}}{\mathrm{Q}^{#1}_{#3}\left(#4\right)}{\mathrm{Q}^{#1}_{#3}\left(#4,#2\right)}
}
}
\def\IF{\mathrm{IF}}
\def\cl{\mathrm{CL}}
\def\CR{\mathrm{CR}}

\usepackage{t1enc}
\newtheorem{theorem}{Theorem}
\newtheorem{lemma}[theorem]{Lemma}

\newtheorem{proposition}[theorem]{Proposition}

\newtheorem{remark}{Remark}

\renewcommand{\hat}{\widehat}
\def\ltwo{\ensuremath{\mathrm{L}^2}}
\def\iid{i.i.d.}

\title[CLT for the robust wavelet regression estimator]{Central limit theorem for the robust log-regression wavelet estimation of the memory parameter in the Gaussian semi-parametric context}
\author{O.~Kouamo}
\author{C.~L\'evy-leduc}
\author{E.~Moulines}
\address{ENSP, LIMSS, BP : 8390 Yaound\'e Cameroun, Institut Telecom/Telecom ParisTech - 46, rue Barrault, 75634
  Paris C\'edex 13, France. }
\email{olaf.kouamo@telecom-paristech.fr}
\address{CNRS/LTCI/Telecom ParisTech - 46, rue Barrault, 75634 Paris C\'edex 13, France.}
\email{celine.levy-leduc@telecom-paristech.fr}
\address{Institut Telecom/Telecom ParisTech - 46, rue Barrault, 75634
  Paris C\'edex 13, France.}
\email{eric.moulines@telecom-paristech.fr}
\date{\today}

\begin{document}
\begin{abstract}

In this paper, we study robust estimators of the memory parameter $d$
of a (possibly) non stationary Gaussian time series with generalized
spectral density $f$. This generalized spectral density is
characterized by the memory parameter $d$ and by a function $f^\ast$ which specifies the
short-range dependence structure of the process.
Our setting is semi-parametric since both $f^\ast$ and $d$ are unknown
and $d$ is the only parameter of interest. The
memory parameter $d$ is estimated by
regressing the logarithm of the estimated variance of the wavelet
coefficients at different scales.
The two  estimators of $d$ that we consider are
based on robust estimators of the variance of the wavelet
coefficients, namely the square of
the scale estimator proposed by \cite{rousseeuw:croux:1993}
and the median of the square of the wavelet coefficients.
We establish a Central Limit Theorem for these robust estimators
as well as for the estimator of $d$ based on the classical estimator of the
variance proposed by \cite{moulines:roueff:taqqu:2007:fractals}.
Some Monte-Carlo experiments are presented to illustrate our claims
and compare the performance of the different estimators. The
properties of the three estimators are also compared on
the Nile River data and the Internet traffic packet counts data. The theoretical results and the empirical evidence strongly suggest
using the robust estimators as an alternative to estimate
the memory parameter $d$ of Gaussian time series.
\end{abstract}
\maketitle
\section{Introduction}\label{sec:intro}

Long-range dependent processes are characterized by hyperbolically slowly decaying correlations or by a spectral density exhibiting a fractional pole
at zero frequency. During the last decades, long-range dependence (and the closely related self-similarity phenomena)
has been observed in many different fields,
including financial econometrics, hydrology or analysis of Internet traffic.  In most of these applications, however,
the presence of atypical observations is quite common. These outliers might be due to gross errors in the observations but also to
unmodeled disturbances; see for example \cite{stoev:taqqu:park:2006} and \cite{stoev:taqqu:2005} for possible explanations of the presence of
outliers in Internet traffic analysis. It is well-known that even a few atypical observations
can severely affect estimators, leading to incorrect conclusions.
Hence, defining robust estimators of the memory parameter which are less sensitive to the
presence of additive outliers is a challenging practical problem.

In this paper, we consider the  class of fractional processes, denoted $M(d)$  defined as follows.
Let  $X= \{X_k \}_{k\in\Zset}$ be a real-valued Gaussian process, not necessarily stationary
and denote by $\diffop X$ the first order difference of $X$, defined by
$[\diffop X]_n= X_n - X_{n-1}$, $n \in \Zset$.
Define, for an integer $K\geq1,$ the $K$-th
order difference recursively as follows : $\diffop^K=\diffop\circ\diffop^{K-1}.$
Let $f^\ast$ be a bounded non-negative symmetric function which is bounded away
from zero in a neighborhood of the origin.
Following \cite{moulines:roueff:taqqu:2007:jtsa}, we say that $X$ is an $M(d)$ process
if for any integer $K>d-1/2$, $\diffop^KX$
is stationary with spectral density function
\begin{equation}\label{eq:spdelta}
f_{\diffop^K X}(\lambda) = |1-\rme^{-\rmi\lambda}|^{2(K-d)}\,f^\ast(\lambda),\quad  \lambda\in(-\pi,\pi) \; .
\end{equation}
Observe that $f_{\diffop^K X}(\lambda)$ in~(\ref{eq:spdelta}) is integrable since $-(K-d)<1/2$. 
When $d\geq1/2$, the process is not stationary. One can nevertheless associate to $X$ the function
\begin{equation}
\label{eq:SpectralDensity:FractionalProcess}
f(\lambda) = |1 - \rme^{-\rmi \lambda}|^{-2d}  f^\ast(\lambda)\; ,
\end{equation}
which is called a \textit{generalized spectral density function}.
In the sequel, we assume that $f^\ast\in\calH(\beta,L)$ with
$0 < \beta \leq 2$ and $L>0$ where $\calH(\beta,L)$ denotes the set of
non-negative and symmetric functions $g$
satisfying, for all  $\lambda\in(-\pi,\pi)$,
\begin{equation}\label{eq:Hbeta}
|g(\lambda)- g(0) | \leq L \, g(0) \, |\lambda|^\beta \eqsp .
\end{equation}
Our setting is semi-parametric in that both $d$ and $f^{\ast}$ in
\eqref{eq:SpectralDensity:FractionalProcess}
are unknown. Here, $f^{\ast}$ can be seen as a nuisance parameter
whereas $d$ is the parameter of interest.
This assumption on $f^\ast$ is  typical in the semi-parametric estimation setting; see for instance
\cite{robinson:1995:GSE} and \cite{moulines:soulier:2003} and the references therein.

Different approaches have been proposed for building robust estimators of the memory parameter for
M($d$) processes in the semi-parametric setting outlined above.
\cite{stoev:taqqu:park:2006} have proposed a robustified wavelet based-regression estimator
developed by \cite{abry:veitch:1998}; the robustification is achieved by replacing the estimation of the wavelet coefficients variance
at different scales by the median of the square of the wavelet coefficients.
Another technique to robustify the wavelet regression technique has been outlined in
\cite{park:park:2009} which consists in regressing the logarithm of the square of the wavelet coefficients
at different scales. \cite{fajardo:reisen:cribari:2009}  proposed a robustified
version of the log-periodogram regression  estimator introduced in \cite{geweke:porter-hudak:1983}. The method replaces
the log-periodogram of the observation by a robust estimator of the spectral density in the neighborhood of the zero frequency,
obtained as the discrete Fourier transform of a robust autocovariance estimator defined in
\cite{ma:genton:2000}; the procedure is appealing and has been found to work well but also lacks theoretical support in the semi-parametric context (note however that the consistency and the asymptotic normality of the robust estimator of the covariance have been discussed in \cite{levy-leduc:boistard:2009}).

In the related context of the estimation of the fractal dimension of locally self-similar Gaussian processes
\cite{coeurjolly:2008b} has proposed a robust estimator of the Hurst coefficient; instead of using the variance of the generalized
discrete variations of the process (which are closely related to the wavelet coefficients, despite the facts that the motivations are quite different), this author proposes to use the empirical quantiles and the trimmed-means. The consistency and asymptotic
normality of this estimator is established for a class of locally self-similar processes, using a Bahadur-type representation of the sample quantile;
see also \cite{coeurjolly:2008a}. \cite{shen:zhu:lee:2007} proposes to replace the classical regression of the wavelet coefficients
by a robust regression approach, based on Huberized M-estimators.

The two robust estimators of $d$ that we propose consist in
regressing the logarithm of robust variance estimators
of the wavelet coefficients of the process $X$ on a  range of scales.
We use as robust variance estimators the square of
the scale estimator proposed by \cite{rousseeuw:croux:1993}
and the square of the \textit{mean absolute deviation} (MAD).
These estimators are a robust alternative to the estimator of $d$ proposed by
\cite{moulines:roueff:taqqu:2007:fractals} which uses the same method
but with the classical variance estimator.
Here, we derive a Central Limit Theorem (CLT) for the two robust estimators
of $d$ and, by the way, we give another methodology for
obtaining a Central Limit Theorem for the estimator of $d$
proposed by \cite{moulines:roueff:taqqu:2007:fractals}.
In this paper, we have also established new results on the empirical process of array of stationary Gaussian processes by extending \cite[Theorem~4
]{arcones:1994} and the Theorem of
\cite{csorgo:mielniczuk:1996} to arrays of stationary Gaussian processes.
These new results were very helpful in establishing the CLT for the three estimators of $d$ that we propose.


The paper is organized as follows. In
Section~\ref{sec:wavelet-setting}, we introduce the wavelet setting
and define the wavelet based regression estimators of
$d$. Section~\ref{sec:sc:properties} is dedicated to the asymptotic
properties of the robust estimators of $d$. In this section, we derive
asymptotic expansions of the wavelet spectrum estimators and provide a CLT for the estimators of $d.$
In Section~\ref{sec:num:exp}, some Monte-Carlo experiments are presented in order to support our theoretical
claims. The Nile River data and two Internet traffic packet
counts datasets collected from  the University of North Carolina,
Chapel are studied as an application in
Section~\ref{sec:nile}.
 Sections~\ref{sec:proofs} and \ref{sec:lemmas}  detail the proofs of the theoretical results stated
in Section ~\ref{sec:sc:properties}. 

\section{Definition of the wavelet-based regression estimators of the memory parameter $d$.}
\label{sec:wavelet-setting}
\subsection{The wavelet setting}

The wavelet setting involves two functions $\phi$ and $\psi$ in
$\ltwo(\mathbb{R})$ and their Fourier transforms
\begin{equation}\label{eq:fourier_psi}
\hat{\phi}(\xi) \eqdef \int_{-\infty}^\infty \phi(t) \rme^{- \rmi \xi t}\,\d t \quad
\text{and}
\quad
\hat{\psi}(\xi) \eqdef \int_{-\infty}^\infty \psi(t) \rme^{- \rmi \xi t}\,\d t\;.
\end{equation}
Assume the following:
\begin{enum_W}
\item\label{item:Wreg} $\phi$ and $\psi$ are compactly-supported, integrable, and $\hat{\phi}(0) = \int_{-\infty}^\infty \phi(t)\,\d t = 1$ and $\int_{-\infty}^\infty \psi^2(t)\,\d t = 1$.
\item\label{item:psiHat}
There exists $\alpha>1$ such that
$\sup_{\xi\in\Rset}|\hat{\psi}(\xi)|\,(1+|\xi|)^{\alpha} <\infty$.
\item\label{item:MVM} The function $\psi$ has  $M$ vanishing moments,
  \ie $ \int_{-\infty}^\infty t^m \psi(t) \,dt=0$ for all $m=0,\dots,M-1$.
\item\label{item:MIM} The function $ \sum_{k\in\Zset} k^m\phi(\cdot-k)$
is a polynomial of degree $m$ for all $m=0,\dots,M-1$.
\end{enum_W}\def\allWA{\ref{item:Wreg}-\ref{item:MIM}}
Condition~\ref{item:psiHat} ensures that the Fourier transform $\hat{\psi}$ decreases quickly
to zero. Condition~\ref{item:MVM} ensures that $\psi$ oscillates and that its
scalar product with continuous-time polynomials up to degree $M-1$ vanishes. It is equivalent to asserting that the first
$M-1$ derivatives of $\hat{\psi}$ vanish at the origin and hence
\begin{equation}
\label{eq:MVM:rob}
|\hat{\psi}(\lambda)|=O(|\lambda|^{M})\; ,\;\text{as}\;\lambda\to0\;.
\end{equation}
Daubechies wavelets (with $M\geq 2$) and the Coiflets satisfy these
conditions, see \cite{moulines:roueff:taqqu:2007:fractals}.
Viewing the wavelet $\psi(t)$ as a basic template, define the family $\{\psi_{j,k}, j \in \Zset, k \in \Zset\}$ of translated and dilated functions
\begin{equation}\label{eq:psiJK}
\psi_{j,k}(t)=2^{-j/2}\,\psi(2^{-j}t-k) ,\quad j\in\Zset,\, k\in\Zset \eqsp .
\end{equation}
Positive values of $k$ translate $\psi$ to the right, negative values to the left. The \emph{scale index} $j$ dilates $\psi$
so that large values of $j$ correspond to coarse scales and hence to low frequencies.
We suppose throughout the paper that
\begin{equation}
\label{eq:ConditionD}
(1+\beta)/2-\alpha<d\leq M\,.
\end{equation}
We now describe how the wavelet coefficients are defined in discrete time, that is for a real-valued sequence
$\{x_k,\,k\in\Zset\}$ and for a finite sample $\{x_k,\,k=1,\dots,n\}$. Using the scaling function $\phi$, we first
interpolate these discrete values to construct the following continuous-time functions
\begin{equation}\label{eq:bX}
\bx_n(t) \eqdef \sum_{k=1}^n x_k \,\phi(t-k) \quad\text{and}\quad \bx(t) \eqdef \sum_{k\in\Zset} x_k\, \phi(t-k), \quad
t\in\Rset \; .
\end{equation}
Without loss of generality we may suppose that the support of the scaling function $\phi$ is included in $[-\L,0]$ for some
integer $\L\geq1$. Then
$$
\bx_n(t)=\bx(t)\quad\text{for all}\quad t\in[0, n-\L+1]\;.
$$
We may also suppose that the support of the wavelet function $\psi$ is included in $[0,\L]$. With these conventions, the support of $\psi_{j,k}$ is included in the
interval $[2^j k, 2^j(k+\L)]$. The wavelet coefficient $\dwt_{j,k}$ at scale $j\geq0$ and location $k\in\Zset$ is formally defined
as the scalar product in $\ltwo(\mathbb{R})$ of the function $t \mapsto \bx(t)$ and the wavelet $t \mapsto \psi_{j,k}(t)$:
\begin{equation}\label{eq:coeffN}
\dwt_{j,k} \eqdef \int_{-\infty}^\infty \bx(t) \psi_{j,k}(t)\,\d t
=\int_{-\infty}^\infty \bx_n(t) \psi_{j,k}(t)\,\d t,
 \quad  j \geq 0, k \in \Zset \;,
\end{equation}
when $[2^j k, 2^j k+\L]\subseteq [0, n-\L+1]$, that is, for all $(j,k)\in\indexset_n$, where
\begin{equation}
\label{eq:deltanrob}
\indexset_n \eqdef \{(j,k):\,j\geq0, 0\leq k \leq n_j-1 \}\quad\text{with}\quad n_j= [2^{-j}(n-\L+1)-\L+1] \eqsp.
\end{equation}

 If $\diffop^MX$ is stationary, then from \cite[Eq~(17)]{moulines:roueff:taqqu:2007:jtsa} the process $\{W_{j,k}\}_{k\in\Zset}$ of wavelet coefficients at scale
$j\geq0$ is stationary but the two--dimensional process $\{[W_{j,k},\,W_{j',k}]^T \}_{k\in \Zset}$ of wavelet
coefficients at scales $j$ and $j'$, with $j \geq j'$, is not stationary. Here  $^T$ denotes the
transposition. This is why we consider instead the stationary  \emph{between-scale} process
\begin{equation}\label{eq:betweenscaleProc}
\{[W_{j,k},\,\bdwt_{j,k}(j-j')^T]^T \}_{k\in \Zset}\;,
\end{equation}
where $\bdwt_{j,k}(j-j')$ is defined as follows:
\begin{equation*}
\bdwt_{j,k}(j-j') \eqdef \left[W_{j',2^{j-j'}k},\,W_{j',2^{j-j'}k+1},\,\dots,
  W_{j',2^{j-j'}k+2^{j-j'}-1}\right]^T.
\end{equation*}
%
For all $j,j'\geq1$, the covariance function of the between scale process is given by
\begin{equation}\label{eq:def:cov:rob}
\PCov(\bdwt_{j,k'}(j-j'),\dwt_{j,k}) = \int_{-\pi}^\pi \rme^{\rmi\lambda(k-k')} \,
\bdens[\phi,\psi]{j,j-j'}{\lambda}{f} \, d\lambda \; ,
\end{equation}
where $\bdens[\phi,\psi]{j,j-j'}{\lambda}{f}$ stands for  the cross-spectral
density function of this process.
For further details, we refer the reader to \cite[Corollary 1]{moulines:roueff:taqqu:2007:jtsa}.
The case $j=j'$ corresponds to the spectral density function of the \emph{within-scale} process $ \{ \dwt_{j,k} \}_{k \in \Zset}$.

In the sequel, we shall use that
the within- and between-scale spectral densities $\bdens[\phi,\psi]{j,j-j'}{\lambda}{d}$ of the process $X$ with
memory parameter $d\in\Rset$ can be approximated by the corresponding
spectral density of the generalized fractional Brownian motion
$B_{(d)}$ defined, for $d\in\Rset$ and $u \in \Nset$,
by
\begin{multline}\label{eq:bDpsi}
\bdensasymp[\psi]{u}{\lambda}{d} = \left[ \bdensasympconv[\psi]{(0)}{u}{\lambda}{d}, \dots, \bdensasympconv[\psi]{(2^u-1)}{u}{\lambda}{d} \right] \\
= \sum_{l\in\Zset} |\lambda+2l\pi|^{-2d}\,\be_{u}(\lambda+2l\pi) \,
\overline{\hat{\psi}(\lambda+2l\pi)}\hat{\psi}(2^{-{u}}(\lambda+2l\pi))\;,
\end{multline}
where,
$$
\be_{u}(\xi) \eqdef 2^{-{u}/2}\, [1, \rme^{-\rmi2^{-u}\xi}, \dots,
\rme^{-\rmi(2^{u}-1)2^{-u}\xi}]^T,\quad\xi\in\Rset\;.
$$
For further details, see \cite[p. 307]{moulines:roueff:taqqu:2007:fractals}.

\subsection{Definition of the robust estimators of $d$}\label{sec:def:d}
Let us now define robust estimators of the memory parameter $d$ of the
M($d$) process $X$ from the observations
$X_1,\dots,X_n$. These estimators are derived from the \cite{abry:veitch:1998} construction,
and consists in regressing estimators of the scale spectrum
\begin{equation}\label{eq:def_sigmaj}
\sigma_j^2\eqdef\PVar(W_{j,0})
\end{equation}
with respect to the scale index $j$. More precisely, if
$\hat{\sigma}_j^2$ is an estimator of $\sigma^2_j$ based on
$W_{j,0:n_j-1}=(W_{j,0},\dots,W_{j,n_j-1})$ then
an estimator of the memory parameter $d$ is obtained by
regressing $\log (\hvj{j}{n_j})$  for a finite number of scale indices
$j \in \{J_0, \dots, J_0+\ell\}$
where $J_0=J_0(n)\geq0$ is the lower scale and $1+\ell\geq2$ is the number of scales in the regression.
 The regression estimator can be expressed formally as
\begin{equation}\label{eq:definition:estimator:regression}
\hat{d}_{n}(J_0,\regressweights) \eqdef \sum_{j=J_0}^{J_0+\ell} w_{j-J_0} \log \left( \hvj{j}{n_j} \right) \eqsp ,
\end{equation}
where the vector $\regressweights \eqdef[w_0,\dots,w_{\ell}]^T$  of weights satisfies
$\sum_{i=0}^{\ell} w_{i}  = 0$ and $2 \log(2) \sum_{i=0}^{\ell} i w_{i}  = 1$,
see \cite{abry:veitch:1998} and  \cite{moulines:roueff:taqqu:2007:jtsa}. For $J_0\geq 1$ and $\ell>1,$ one may choose for example \textbf{w} corresponding to the least squares regression matrix, defined by $\textbf{w}=DB(B^TDB)^{-1}\mathbf{b}$ where
\begin{equation*}
  \label{eq:bAndB1}
\mathbf{b}\eqdef \left[\begin{matrix}
0&(2\log(2))^{-1}
\end{matrix}
\right],
\quad
B \eqdef \left[\begin{matrix}
1 & 1 & \dots & 1 \\
0 & 1 & \dots & \ell
\end{matrix}\right]^T
\end{equation*}
is the design matrix and $D$ is an arbitrary positive definite matrix. The best choice of $D$ depends on the memory parameter $d$. However a good approximation of this optimal matrix $D$ is the diagonal matrix with diagonal entries $D_{i,i}=2^{-i},$ $i=0\dots,\ell$; see \cite{fay:moulines:roueff:2008}
and the references therein. We will use this choice of the design matrix in the numerical experiments.
A heuristic justification for this choice is that
by \cite[Eq. (28)]{moulines:roueff:taqqu:2007:fractals},
\begin{equation}\label{eq:approx_sigmaj}
\sigma_j^2\sim C\, 2^{2jd}\; ,\textrm{ as }j\to\infty\;,
\end{equation}
where $C$ is a positive constant.

In the sequel, we shall consider three different estimators of $d$
based on three different estimators of the scale spectrum $\sigma_j^2$
with respect to the scale index $j$ which are defined below.

\subsubsection{Classical scale estimator}
This estimator has  been considered in the original contribution of \cite{abry:veitch:1998}
and consists in estimating the scale spectrum $\sigma^2_j$ with respect to the scale index $j$
by the empirical variance
\begin{equation}\label{eq:def:cl}
\hat{\sigma}^2_{\cl,j}=\frac{1}{n_j}\sum_{i=1}^{n_j}W^2_{j,i}\;,
\end{equation}
where for any $j$, $n_j$ denotes the number of available wavelet
coefficients at scale index $j$ defined in \eqref{eq:deltanrob}.

\subsubsection{Median absolute deviation}
This estimator is well-known to be a robust estimator of the scale and
as mentioned by \cite{rousseeuw:croux:1993} it has several appealing
properties: it is easy to compute and has the best possible breakdown
point (50\%).
Since the wavelet coefficients $W_{j,i}$ are centered Gaussian
observations, the square of the median absolute deviation of $W_{j,0:n_j-1}$ is
defined by
\begin{equation}
\hat{\sigma}^2_{\MAD,j}=\left(m(\Phi)\underset{0\leq
      i\leq n_j-1}{\med} |W_{j,i}|\right)^2\;,
\label{eq:def:mad}
\end{equation}
where  $\Phi$ denotes the c.d.f of a standard Gaussian random variable and
\begin{equation}\label{eq:mphi}
m(\Phi)=1/\Phi^{-1}(3/4)=1.4826\;.
\end{equation}
The use of the median estimator to estimate the scalogram has been suggested
 to estimate the memory parameter in \cite{stoev:taqqu:park:marron:2005} (see also \cite[p.~420]{percival:walden:2006}). A closely related technique is considered in \cite{coeurjolly:2008a} and \cite{coeurjolly:2008b} to estimate the Hurst coefficient of locally self-similar Gaussian processes. Note that
 the use of the median of the squared wavelet coefficients has been
 advocated to estimate the variance at a given scale in wavelet
 denoising applications; this technique is mentioned in
 \cite{donoho:johnstone:1994} to estimate 
the scalogram of the noise in the \iid\ context;  \cite{johnstone:silverman:1997} proposed
 to use this method in the long-range dependent context; 
the use of these estimators has not been however rigorously justified.

\subsubsection{The Croux and Rousseeuw estimator}

This estimator is another robust scale estimator introduced in
\cite{rousseeuw:croux:1993}. Its asymptotic properties in several
dependence contexts have been further studied in
\cite{levy-leduc:boistard:2009} and the square of this estimator is defined by
\begin{equation}
\label{eq:def:cr}
\hat{\sigma}^2_{\CR,j} = \left(c(\Phi)
  \{|W_{j,i}-W_{j,k}|;\ 0 \leq i,k \leq n_j-1\}_{(k_{n_{j}})}\right)^2\; ,
\end{equation}
where $c(\Phi)=2.21914$ and $k_{n_j} = \lfloor n_j^2/4 \rfloor$.
That is, up to the multiplicative constant $c(\Phi)$,
$\hat{\sigma}_{\CR,j}$ is the $k_{n_j}$th order statistics of the
$n_j^2$ distances $|W_{j,i}-W_{j,k}|$ between all the pairs of
observations.

\section{Asymptotic properties of the robust estimators of $d$}\label{sec:sc:properties}
\subsection{Properties of the scale spectrum estimators}
The following proposition gives an asymptotic expansion for
$\hat{\sigma}^2_{\cl,j}$,
$\hat{\sigma}^2_{\MAD,j}$ and
$\hat{\sigma}^2_{\CR,j}$
defined in
(\ref{eq:def:cl}),
(\ref{eq:def:mad}) and
(\ref{eq:def:cr}), respectively.
 These asymptotic expansions are used for deriving Central Limit
 Theorems for the different estimators of $d$.

\begin{proposition}\label{lemma:asymp:exp}
Assume that $X$ is a Gaussian $M(d)$ process with generalized spectral
density function defined in \eqref{eq:SpectralDensity:FractionalProcess} such that $f^\ast\in\mathcal{H}(\beta,L)$ for
some $L>0$ and $0<\beta\leq 2$.  Assume that
\ref{item:Wreg}-\ref{item:MIM} hold with $d,$ $\alpha$ and $M$
satisfying \eqref{eq:ConditionD}.
Let $W_{j,k}$ be the wavelet coefficients associated to $X$ defined
by (\ref{eq:coeffN}). If $n \mapsto J_0(n)$ is an integer valued sequence satisfying $J_0(n)\to\infty$ and
$n2^{-J_0(n)}\to\infty$, as $n\to\infty$, then $\hat{\sigma}^2_{\ast,j}$ defined in (\ref{eq:def:cl}),
(\ref{eq:def:mad}) and (\ref{eq:def:cr}), satisfies
the following asymptotic expansion, as $n\to\infty$, for any given $\ell \geq 1$
\begin{align}\label{exp}
\max_{J_0(n) \leq j \leq J_0(n)+\ell} \left| 
\sqrt{n_j}(\hat{\sigma}^2_{\ast,j}-\sigma^2_j)-\frac{2\sigma_j^2}{\sqrt{n_j}}\sum_{i=0}^{n_j-1}
\IF\left(\frac{W_{j,i}}{\sigma_j},\ast,\Phi\right)\right| = o_P(1)\;,
\end{align}
where $\ast$ denotes $\cl$, $\CR$ and $\MAD$, $\sigma_j^2$ is defined in
(\ref{eq:def_sigmaj}) and $\IF$ is given by
\begin{align}
\label{IF:cl}
&\IF\left(x,\cl,\Phi\right)=\frac{1}{2}H_2(x),\\
\label{IF:Q}
&\IF\left(x,\CR,\Phi\right)=c(\Phi)\left(\frac{1/4-\Phi(x+1/c(\Phi))
+\Phi(x-1/c(\Phi))}{\int_{\Rset} \varphi(y)\varphi(y+1/c(\Phi))\rmd y}\right)\;, \\\label{IF:mad}
&\IF(x,\MAD,\Phi)=-m(\Phi)\left(\frac{\left(\1_{\{x\leq 1/m(\Phi)\}}-3/4\right)-\left(\1_{\{x\leq -1/m(\Phi)\}}-1/4\right)}{2\varphi(1/m(\Phi))}\right)\;,
\end{align}
where $\varphi$ denotes the p.d.f of the standard Gaussian random
variable, $m(\Phi)$ and $c(\Phi)$ being defined in (\ref{eq:mphi}) and
(\ref{eq:def:cr}), respectively and $H_2(x)=x^2-1$ is the second Hermite polynomial.
\end{proposition}
The proof is postponed to Section~\ref{sec:proofs}.


We deduce from Proposition \ref{lemma:asymp:exp} and Theorem
\ref{theo:ext:mult:arcones} given and proved in Section~\ref{sec:proofs} the following
multivariate Central Limit Theorem for the wavelet coefficient scales.
\begin{theorem}\label{theo:clt:wavelet}
 Under the assumptions of Proposition \ref{lemma:asymp:exp},
$(\hvj{\ast,J_0}{},\dots,\hvj{\ast,J_0+\ell}{})^T$,
where $\hvj{\ast,j}{}$ is defined in
(\ref{eq:def:cl}), (\ref{eq:def:mad}) and (\ref{eq:def:cr}),
satisfies the following multivariate Central Limit Theorem
\begin{equation}
\label{eq:JointCentralLimitEmpVar}
\sqrt{n2^{-J_0}}2^{-2J_0d} \left(\left[
\begin{array}{c}
\hvj{\ast,J_0}{}\\
\hvj{\ast,J_0+1}{}\\
\vdots\\
\hvj{\ast,J_0+\ell}{}
\end{array}
\right] -
\left[
\begin{array}{c}
\sigma^2_{\ast,J_0}\\
\sigma^2_{\ast,J_0+1}\\
\vdots\\
\sigma^2_{\ast,J_0+\ell}\\
\end{array}
\right] \right) \cd \calN\left(0, \mathbf{U}_{\ast}(d) \right) \;,
\end{equation}
where
\begin{multline}\label{eq:var:approx}
\mathbf{U}_{\ast,i,j}(d)= 4 (f^{\ast}(0))^2 \sum_{p\geq2}\frac{c_p^2(\IF_\ast)}{p!\,\Kvar[\psi]{d}^{p-2}}
\, 2^{d(2+p)i\vee j}2^{d(2-p)i\wedge j+i\wedge j}\\
\times \sum_{\tau\in\Zset}\sum_{r=0}^{2^{|i-j|}-1
}\Big(\int_{-\pi}^{\pi}\bdensasympconv[\psi]{(r)}{|i-j|}{\lambda}{d}\rme^{\rmi\lambda\tau}\d\lambda\Big)^p
\;,\; 0\leq i,j\leq \ell\;.
%
\end{multline}
In \eqref{eq:var:approx},
$\Kvar{d}\eqdef\int_\Rset|\xi|^{-2d}|\hat{\psi}(\xi)|\d\xi,$ $\bdensasymp[\psi]{|i-j|}{\cdot}{d}$ is the cross-spectral density defined in \eqref{eq:bDpsi},
$c_p(\IF_\ast)=\PE[\IF(X,\ast,\Phi)H_p(X)]$, where $H_p$ is the
$p$th Hermite polynomial and $\IF(\cdot,\ast,\Phi)$ is defined
in (\ref{IF:cl}), (\ref{IF:Q}) and (\ref{IF:mad}).
\end{theorem}
The proof of Theorem~\ref{theo:clt:wavelet} is postponed to Section~\ref{sec:proofs}.
\begin{remark}\label{remark:efficiency}
Since for $\ast=\cl$, $\IF(\cdot)=H_2(\cdot)/2$, Theorem
\ref{theo:clt:wavelet} gives an alternative proof to \cite[Theorem 2]{moulines:roueff:taqqu:2007:fractals}
of the limiting covariance matrix of
$(\hvj{\cl,J_0}{},\dots,\hvj{\cl,J_0+\ell}{})^T$
which is given, for $0\leq i,j\leq \ell$, by
$$\mathbf{U}_{\cl,i,j}(d)=4\pi\left(f^\ast(0)\right)^2 2^{4d(i\vee j)+i\wedge
  j}
\int_{-\pi}^{\pi}|\bdensasymp[\psi]{|i-j|}{\lambda}{d}|^2\d\lambda\;.$$
Thus, for $\ast=\CR$ and $\ast=\MAD$, we deduce the following
\begin{equation}\label{eq:eff_rel}
\frac{\mathbf{U}_{\cl,i,i}(d)}{\mathbf{U}_{\ast,i,i}(d)}\geq
\frac{1/2}{\PE\left[\IF_\ast^2(Z)\right]}\;,
\end{equation}
where $Z$ is a standard Gaussian random variable.
With Lemma \ref{lemma:if}, we deduce from the inequality
\eqref{eq:eff_rel} that the asymptotic relative efficiency
of $\hat{\sigma}^2_{\ast,j}$ is larger than 36.76$\%$ when
$\ast=\MAD$ and larger than  82.27$\%$ when $\ast=\CR$.
\end{remark}

\subsection{CLT for the robust wavelet-based regression estimator}


Based on the results obtained in the previous section, we
derive a Central Limit Theorem for the robust wavelet-based regression
estimators of $d$ defined by

\begin{equation}\label{eq:hatd:clcrmad}
 \hat{d}_{\ast,n}(J_0,\regressweights) \eqdef \sum_{j=J_0}^{J_0+\ell} w_{j-J_0} \log \left( \hat{\sigma}^2_{\ast,j} \right)\;,
 \end{equation}
where $\hat{\sigma}^2_{\ast,j}$ are given for
$\ast=\cl$, $\MAD$ and $\CR$ by (\ref{eq:def:cl}), (\ref{eq:def:mad})
and (\ref{eq:def:cr}), respectively.

\begin{theorem}\label{theo:clt:hatd}
Under the same assumptions as in Proposition~\ref{lemma:asymp:exp} and if
\begin{align}\label{eq:condition:alpha:beta}
n2^{-(1+2\beta)J_0(n)}\to0\;, \textrm{ as } n\to\infty,
\end{align}
then, $\hat{d}_{\ast,n}(J_0,\regressweights)$ satisfies the following
Central Limit Theorem:
\begin{equation}
\label{eq:clt:hatd}
\sqrt{n2^{-J_0(n)}}\left(\hat{d}_{\ast,n}(J_0,\regressweights)-d\right)\cd\mathcal{N}\left(0,\regressweights^T\AVvar{d}{\ast}\regressweights\right)\,,
\end{equation}
where $\AVvar{d}{\ast}$ is the $(1+\ell)\times(1+\ell)$ matrix defined by
\begin{multline}
\label{eq:definitionAVvar}
\AVvar{d}{\ast,i,j} =
\sum_{p\geq2}\frac{4c_p^2(\IF_\ast)}{p!\,\Kvar[\psi]{d}^{p}}
 2^{pd|i-j|+i\wedge j}
\sum_{\tau\in\Zset}\sum_{r=0}^{2^{|i-j|}-1
}\Big(\int_{-\pi}^{\pi}\bdensasympconv[\psi]{(r)}{|i-j|}{\lambda}{d}\rme^{\rmi\lambda\tau}\d\lambda\Big)^p
\;,\; 0\leq i,j\leq \ell\;.
\end{multline}
In \eqref{eq:definitionAVvar},
$\Kvar{d}=\int_\Rset|\xi|^{-2d}|\hat{\psi}(\xi)|\d\xi,$ $\bdensasymp[\psi]{|i-j|}{\cdot}{d}$ is the cross-spectral density defined in \eqref{eq:bDpsi},
$c_p(\IF_\ast)=\PE[\IF(X,\ast,\Phi)H_p(X)]$, where $H_p$ is the
$p$th Hermite polynomial and $\IF(\cdot,\ast,\Phi)$ is defined
in (\ref{IF:cl}), (\ref{IF:Q}) and (\ref{IF:mad}).
\end{theorem}
The proof of Theorem~\ref{theo:clt:hatd} is a straightforward
consequence of
\cite[Proposition~3]{moulines:roueff:taqqu:2007:fractals} and
Theorem~\ref{theo:clt:wavelet} and is thus not detailed here.
%
%
\begin{remark}
Since it is difficult to provide a theoretical lower bound for the
asymptotic relative efficiency (ARE) of
$\hat{d}_{\ast,n}(J_0,\regressweights)$
defined by
\begin{equation}\label{eq:ARE}
\textrm{ARE}_\ast(d)=\regressweights^T\AVvar{d}{\cl}\regressweights/\regressweights^T\AVvar{d}{\ast}\regressweights\;,
\end{equation}
where $\ast=\CR$ or $\MAD$, we propose to compute
this quantity empirically. We know from Theorem \ref{theo:clt:hatd}
that the expression of the limiting covariance matrix $\AVvar{d}{\ast,i,j}$ is valid
for all Gaussian $M(d)$ processes satisfying the assumptions given in
Proposition \ref{lemma:asymp:exp}, thus it is enough to compute
$\textrm{ARE}_\ast(d)$ 
in the particular case of a Gaussian ARFIMA(0,$d$,0)
process $(X_t)$. Such a process is defined by
\begin{equation}\label{e:FARIMA}
X_t=(I-B)^{-d} Z_t=\sum_{j\geq
  0}\frac{\Gamma(j+d)}{\Gamma(j+1)\Gamma(d)}Z_{t-j}\; ,
\end{equation}
where $\{Z_t\}$ are i.i.d $\mathcal{N}(0,1)$. We propose to evaluate
$\textrm{ARE}_\ast(d)$
when $d$ belongs to $[-0.8;3]$. 
With such a choice of $d$, both stationary and non-stationary
processes are considered. The empirical values of
$\textrm{ARE}_\ast(d)$ are given in Table~\ref{table:eff:d}.
The results were obtained from the observations $X_1,\dots,X_n$
where $n=2^{12}$ and 1000 independent replications.
We used Daubechies wavelets with $M=2$ vanishing moments when $d\leq 2$ and $M=4$ when $d>2$
which ensures that condition \eqref{eq:ConditionD} is satisfied.
The smallest scale is chosen to be $J_0=3$ and $J_0+\ell=8$.
{\footnotesize
\begin{table}[!h]
\centering
\begin{tabular}{ccccccccccccccc}
\hline\hline
				$d$ &-0.8&-0.4 &-0.2&0& 0.2 &0.6& 0.8&1 &1.2&1.6& 2 &2.2&2.6&3 \\
				\hline
        $\text{ARE}_{\CR}(d)$&0.72 &0.67 & 0.63& 0.65& 0.70&0.63 &0.70  &0.75& 0.76& 0.75& 0.79&0.74 &0.77 &  0.74 \\
        $\text{ARE}_{\MAD}(d)$ & 0.48& 0.39& 0.38& 0.36& 0.43& 0.39& 0.44 &0.47& 0.45& 0.50& 0.48& 0.5 &0.49 & 0.49\\
        \hline
        \hline
        \end{tabular}
        \caption{{\footnotesize Asymptotic relative efficiency of $\hat{d}_{n,\CR}$
          and $\hat{d}_{n,\MAD}$ with respect to $\hat{d}_{n,\cl}$.}}
\label{table:eff:d}
\end{table}
}

From Table \ref{table:eff:d}, we can see that $\hat{d}_{n,\CR}$ is
more efficient than $\hat{d}_{n,\MAD}$ and that its asymptotic
relative efficiency $\text{ARE}_{\CR}$ ranges from 0.63 to 0.79.
These results indicate empirically that the the loss of efficiency of the robust estimator
$\hat{d}_{n,\CR}$ is moderate and makes it
an attractive robust procedure to the non-robust estimator  $\hat{d}_{n,\cl}$.
\end{remark}

\section{Numerical experiments}\label{sec:num:exp}
In this section the robustness
properties of the different estimators of $d$, namely
$\hat{d}_{\cl,n}(J_0,\regressweights)$,
$\hat{d}_{\CR,n}(J_0,\regressweights)$ and
$\hat{d}_{\MAD,n}(J_0,\regressweights)$, that are defined
in Section \ref{sec:def:d} are investigated using Monte Carlo experiments.
In the sequel, the memory parameter
$d$ is estimated from $n=2^{12}$ observations of a Gaussian ARFIMA(0,$d$,0)
process defined in (\ref{e:FARIMA})
when $d$=0.2 and 1.2 eventually corrupted by additive outliers.
We use the Daubechies wavelets with $M=2$ vanishing moments which ensures that condition \eqref{eq:ConditionD} is
satisfied.

Let us first explain how to choose the parameters $J_0$ and
$J_0+\ell$. With $n=2^{12}$, the maximal
available scale is equal to 10. 
Choosing $J_0$ too small may introduce a bias
in the estimation of $d$ by Theorem \ref{theo:clt:hatd}. However,
at coarse scales (large values of $J_0$), the number of observations
may be too small and thus choosing $J_0$ too large may yield a large
variance. 
Since at scales $j=9$ and $j=10,$ we have respectively 5 and 1
observations, we chose $J_0+\ell=8.$ For the choice of $J_0$, we 
proposed to use the empirical rule illustrated in Figure~\ref{fig:std:d}.
In this figure, we display the estimates $\hat{d}_{n,\cl}$,
$\hat{d}_{n,\CR}$ and $\hat{d}_{n,\MAD}$ of the memory parameter $d$
as well as their respective 95$\%$ confidence intervals from $J_0=1$
to $J_0=7$ with $J_0+\ell=8$. We propose to choose $J_0=3$
in both cases ($d=0.2$ and $d=1.2$) since the successive confidence
intervals starting from $J_0=3$ to $J_0=7$ are such that the smallest
one is included in the largest one. We shall take $J_0=3$
in the sequel.
\begin{figure}[!h]
\begin{tabular}{cc}
\includegraphics[width=0.4\textwidth,angle=-90]{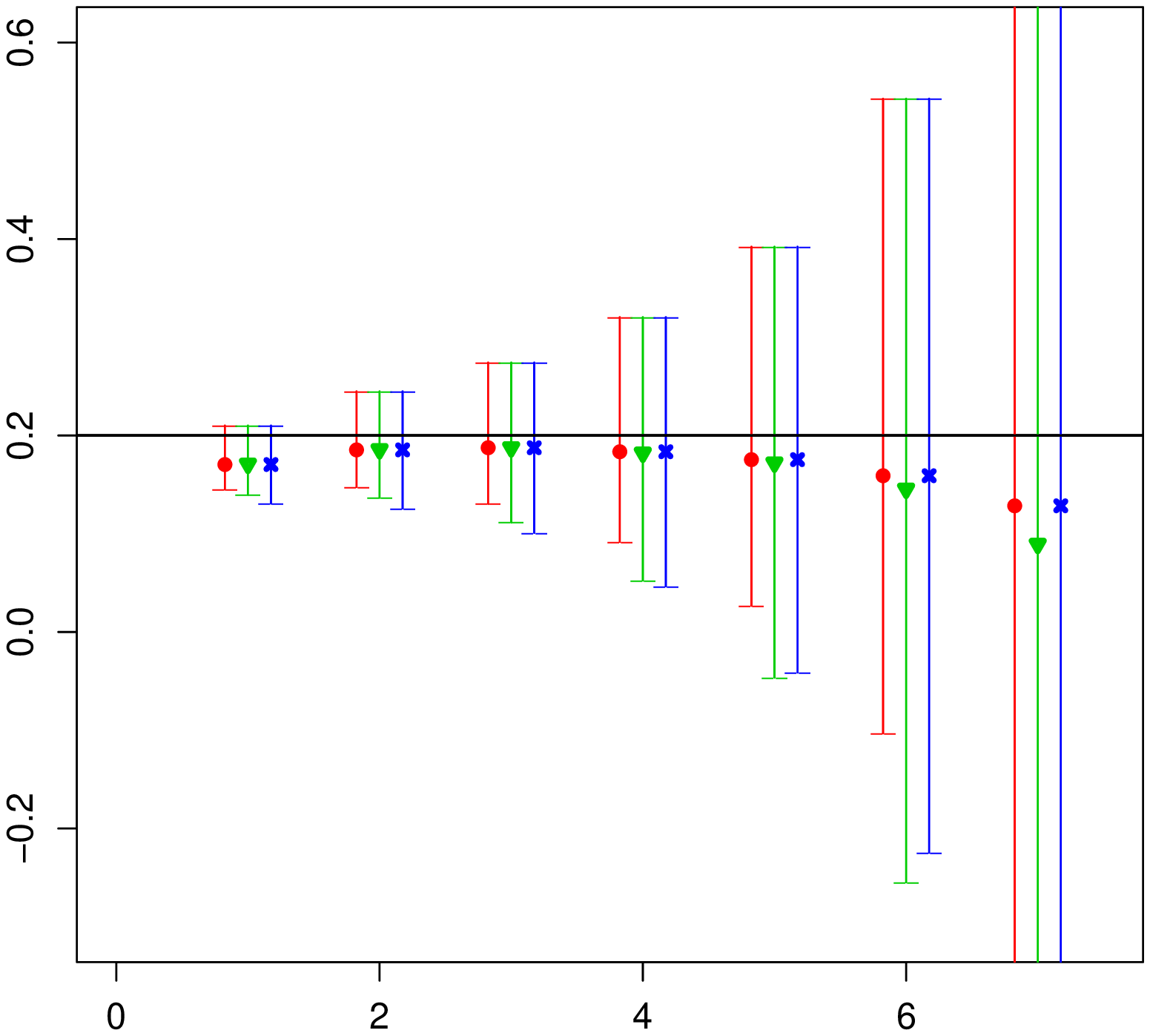}
&\includegraphics[width=0.4\textwidth,angle=-90]{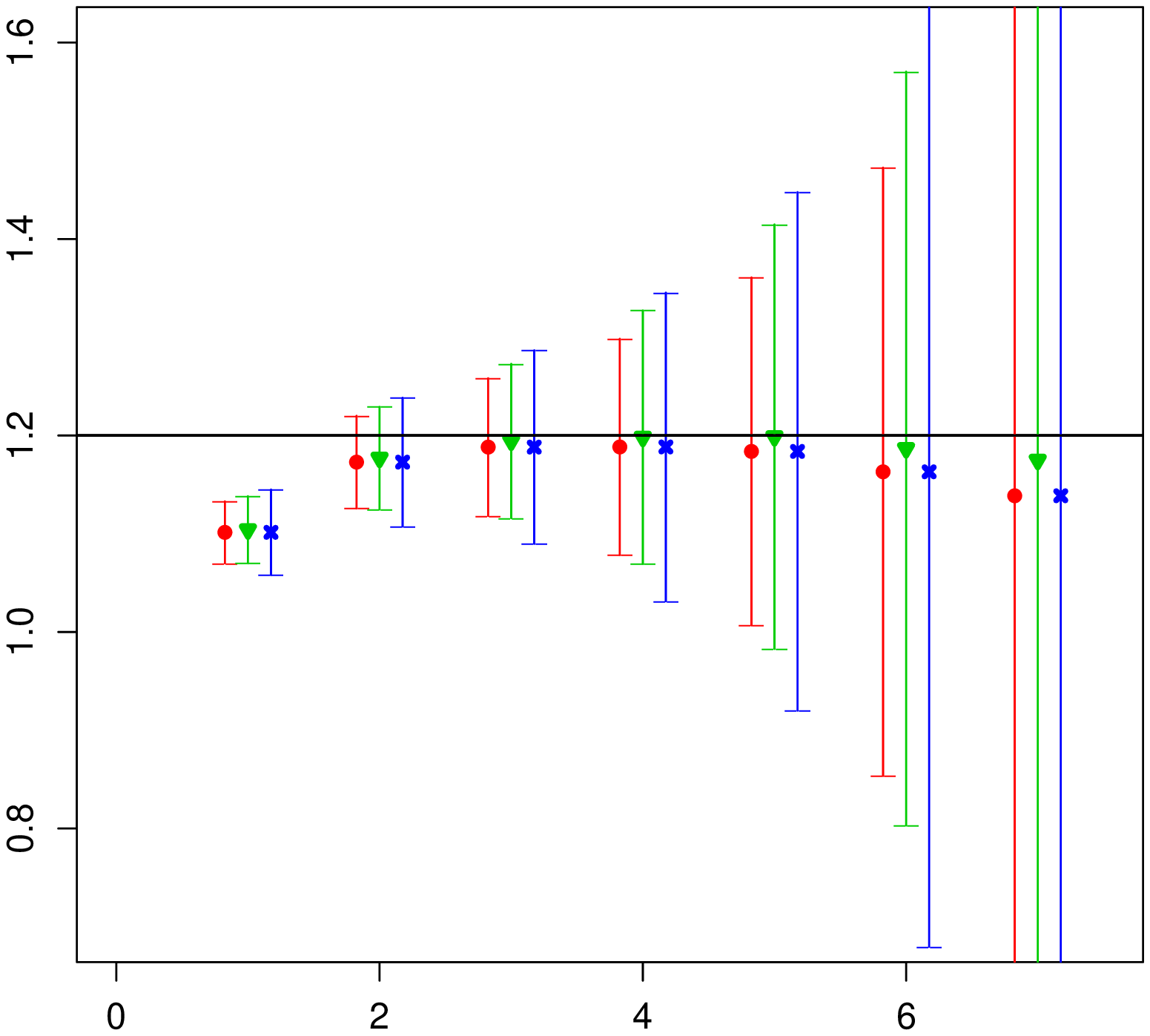}
\end{tabular}
\vspace{-1cm}
\caption{\footnotesize{Confidence intervals of the estimates $\hat{d}_{n,\cl}$, $\hat{d}_{n,\CR}$ and $\hat{d}_{n,\MAD}$ of an ARFIMA($0,d,0$) process
with $d=0.2$ (left) and $d=1.2$ (right) for $J_0=1,\dots,8$
and $J_0+\ell=9$.
For each $J_0$, are displayed confidence interval 
associated to
$\hat{d}_{n,\cl}$ (red), interval$\hat{d}_{n,\CR}$ (gren) and $\hat{d}_{n,\MAD}$ (blue), respectively.}}
\label{fig:std:d}
\end{figure}

%
 In the left panels of Figures \ref{fig:clt_02} and \ref{fig:clt_12}  the empirical
  distribution of  $\sqrt{n2^{-J_0}}(\hat{d}_{\ast,n}-d)$
 are displayed when $\ast=\cl,\MAD$ and $\CR$ for the ARFIMA(0,$d$,0) model with
$d=0.2$ (Figure~\ref{fig:clt_02}) and $d=1.2$ (Figure~\ref{fig:clt_12}), respectively. They were computed using 5000
replications; their shapes are close to the Gaussian density (the standard deviations are of course different).
In the right panels of Figures \ref{fig:clt_02} and \ref{fig:clt_12},  the empirical
  distribution of  $\sqrt{n2^{-J_0}}(\hat{d}_{\ast,n}-d)$ are displayed  when
outliers are present. We introduce $1\%$ of additive outliers in the
observations; these outliers are obtained by choosing uniformly at random a time index and by
adding to the selected observation 5 times the standard error of the raw observations.
The  empirical distribution of $\sqrt{n2^{-J_0}}(\hat{d}_{\cl,n}-d)$  is
clearly located far away from zero especially in the non stationary ARFIMA($0,1.2,0$) model. One can also observe the
considerable increase in the variance of the classical estimator.
In sharp contrast, the distribution of the robust estimators
$\sqrt{n2^{-J_0}}(\hat{d}_{\MAD,n}-d)$ and
$\sqrt{n2^{-J_0}}(\hat{d}_{\CR,n}-d)$ stays symmetric and the variance stays constant.
\begin{figure}[!h]
\begin{tabular}{cc}
\includegraphics*[width=0.4\textwidth]{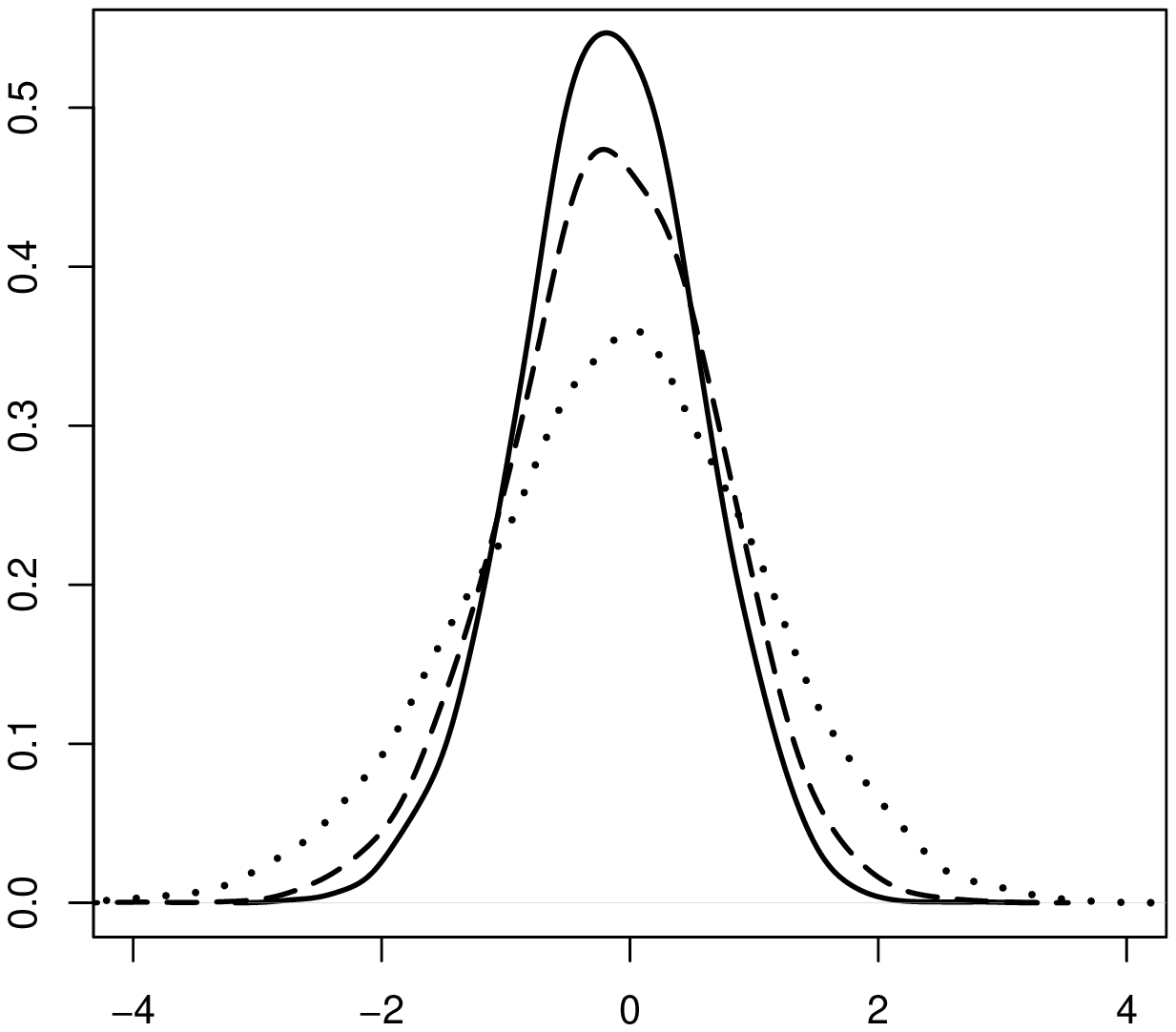}
&\includegraphics*[width=0.4\textwidth]{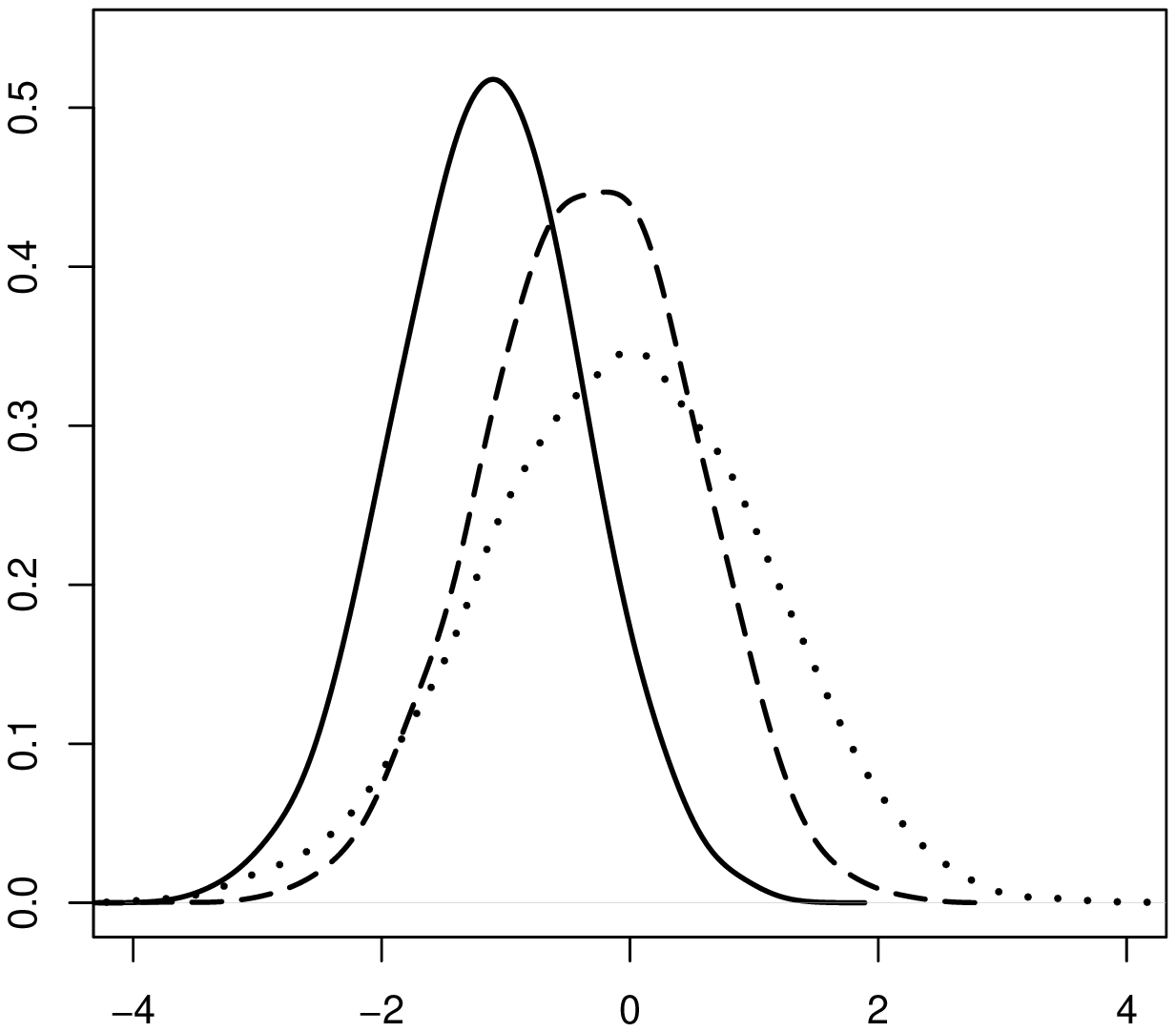}
\end{tabular}
\vspace{-1cm}
\caption{\footnotesize{Empirical densities of the
    quantities $\sqrt{n 2^{-J_0}}(\hat{d}_{\ast,n}-d)$, with
    $\ast=\cl$ (solid line), $\ast=\CR$ (dashed line) and $\ast=\MAD$
    (dotted line) of the ARFIMA(0,0.2,0) model without outliers (left) and
    with 1$\%$ of outliers (right).}}
\label{fig:clt_02}
\end{figure}
\begin{figure}[!h]
\begin{tabular}{cc}
\includegraphics*[width=0.4\textwidth]{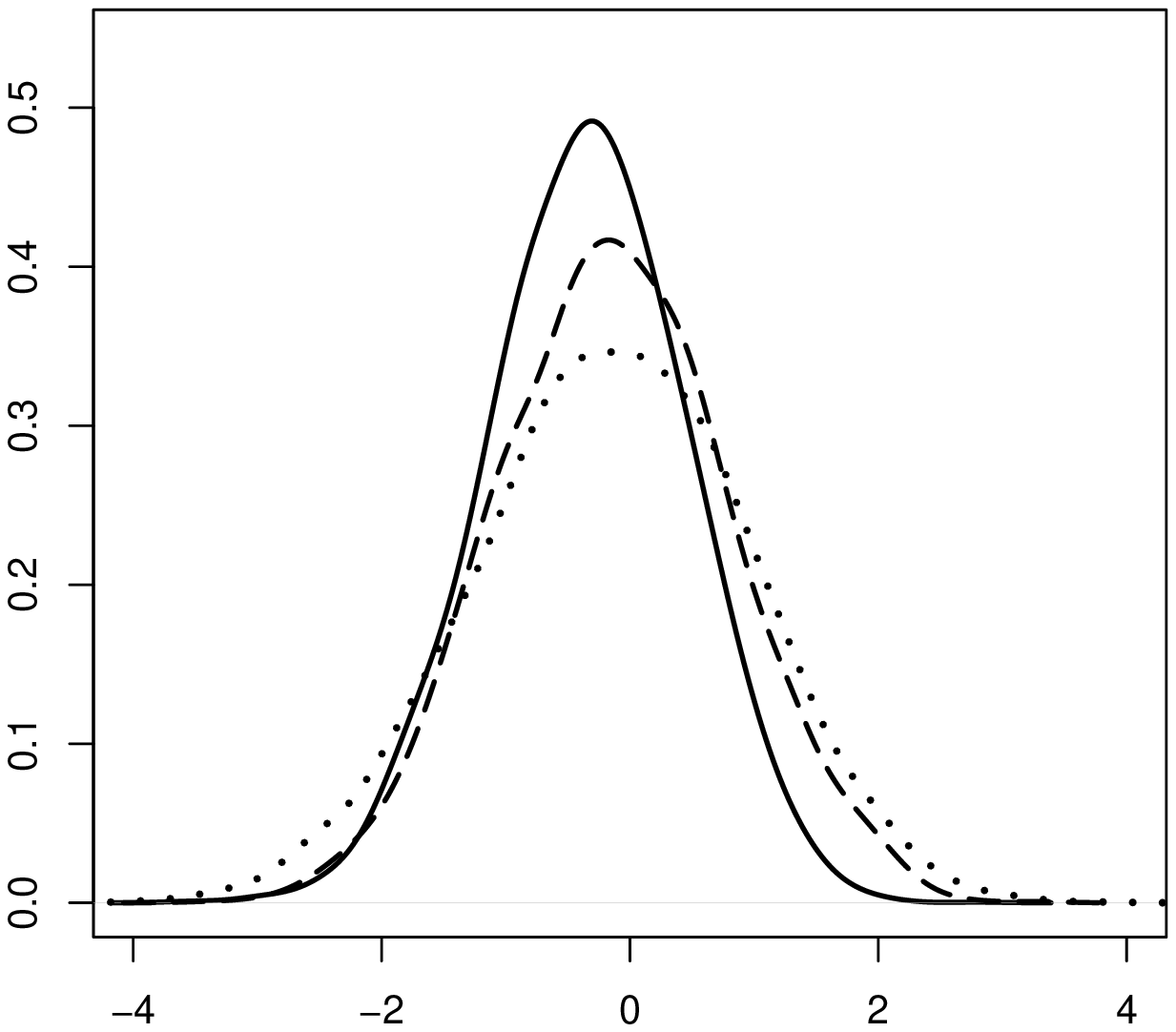}
&\includegraphics*[width=0.4\textwidth]{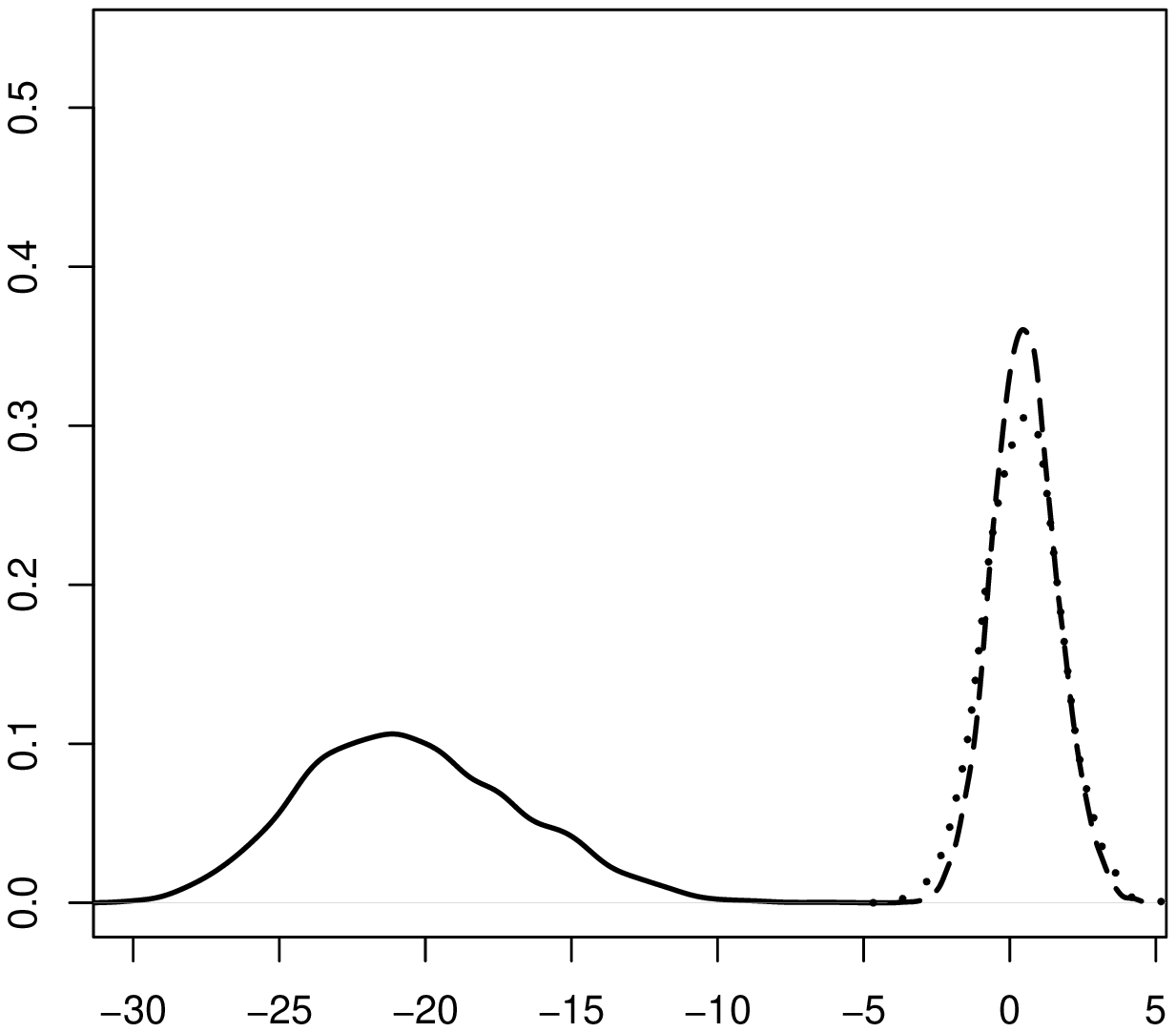}
\end{tabular}
\vspace{-1cm}
\caption{\footnotesize{Empirical densities of the
    quantities $\sqrt{n 2^{-J_0}}(\hat{d}_{\ast,n}-d)$, with
    $\ast=\cl$ (solid line), $\ast=\CR$ (dashed line) and $\ast=\MAD$
    (dotted line) of the ARFIMA(0,1.2,0) model without outliers (left) and
    with 1$\%$ of outliers (right).}}
\label{fig:clt_12}
\end{figure}
\section{Application to real Data}\label{sec:nile}

In this section, we compare the performance of
the different estimators of the long memory parameter $d$ introduced
in Section \ref{sec:def:d} on two different real datasets.

\subsection{Nile River data}
The Nile River dataset is a well-known time series, which has
been extensively analyzed; see \cite[Section~1.4,p.~20]{beran:1994}.
The data consists of yearly minimal water levels of the Nile river measured at the
Roda gauge, near Cairo,  for the years 622--1284~AD and contains 663 observations;
The units for the data as presented by \cite{beran:1994} are
centimeters.
The empirical mean and the standard deviation of the data are equal to 1148 and 89.05, respectively.
The question has been raised as to whether the Nile time series
contains  outliers; see for example \cite{beran:1992},
\cite{robinson:1995:GSE}, \cite{chareka:matarise:turner:2006} and \cite{fajardo:reisen:cribari:2009}.
The test procedure developed by \cite{chareka:matarise:turner:2006}
suggests the presence of outliers at 646 AD ($p$-value 0.0308) and at
809 ($p$-value 0.0007). Another possible outliers is at 878 AD.
Since the number of observations is small, in the estimation of $d,$ we took $J_0=1$ and $J_0+\ell=6$. With
this choice, we observe a significant difference between the classical estimators $\hat{d}_{n,\cl}=0.28$ (with 95\% confidence interval [0.23, 0.32]) and the robust estimators $\hat{d}_{n,\CR}=0.408$ (with 95\% confidence interval [0.34, 0.46]) and $\hat{d}_{n,\MAD}=0.414$ (with 95\% confidence interval [0.34, 0.49]).
Thus, to better
understand the influence of outliers on the estimated memory parameter
in practical situations, a new dataset with artificial outliers was
generated.
Here, we replaced the presumed outliers of
\cite{chareka:matarise:turner:2006}
by the value of the observation plus 10 times the standard deviation.
The new memory parameter estimators are $\hat{d}_{n,\cl}=0.12$,
$\hat{d}_{n,\CR}=0.4$
and $\hat{d}_{n,\MAD}=0.392.$ As was expected, the values of the
robust estimators remained stable.
However, the classical estimator of $d$ was significantly affected.
A robust estimate of $d$ for the Nile data is also given in
\cite{agostinelli:bisaglia:2003} and in
\cite{fajardo:reisen:cribari:2009}.
The authors found 0.412 and 0.416, respectively.
These values are very close to $\hat{d}_{n,\CR}=0.408$ and $\hat{d}_{n,\MAD}=0.414$.

\subsection{Internet traffic packet counts data}

In this section,  two Internet traffic packet
counts datasets collected at  the University of North Carolina,
Chapel (UNC) are analyzed. These datasets are available from the website
http://netlab.cs.unc.edu/public/old\_research/net\_lrd/.
These datasets have been studied by \cite{park:park:2009}.

Figure~\ref{fig:sat1930} (left) displays a packet count time series
 measured at the link of UNC on April 13, Saturday, from 7:30
p.m. to 9:30 p.m., 2002 (Sat1930). Figure~\ref{fig:sat1930} (right) displays
the same type of time series but on April 11, a Thursday, from 1 p.m. to 3 p.m., 2002 (Thu1300).
These packet counts were measured every 1 millisecond but, for a better
display, we aggregated them at 1 second.

\begin{figure}[!ht]\vspace{-.41cm}
\begin{tabular}{cc}
\includegraphics*[width=0.4\textwidth]{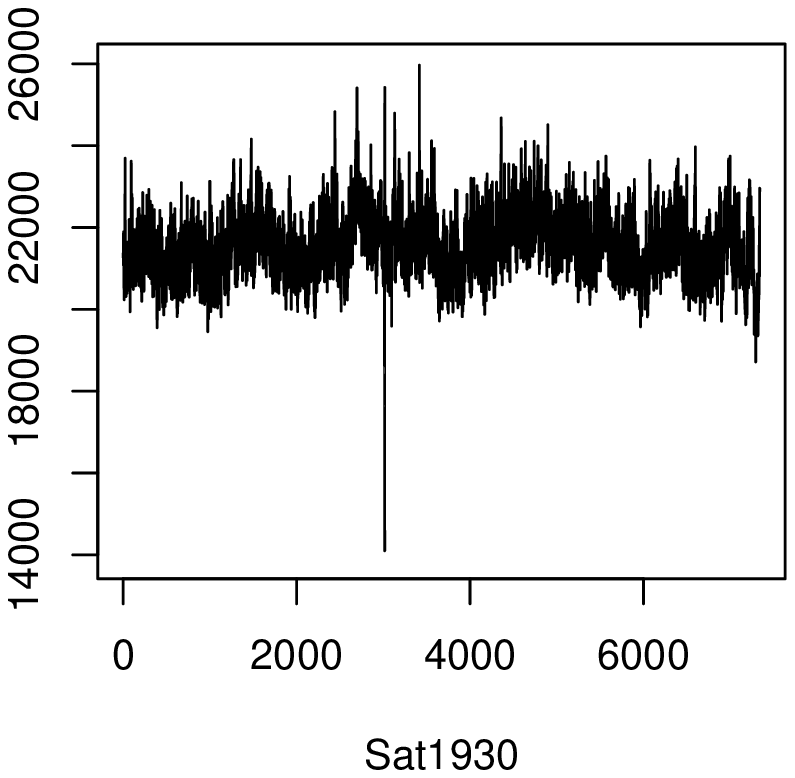}
\includegraphics*[width=0.4\textwidth]{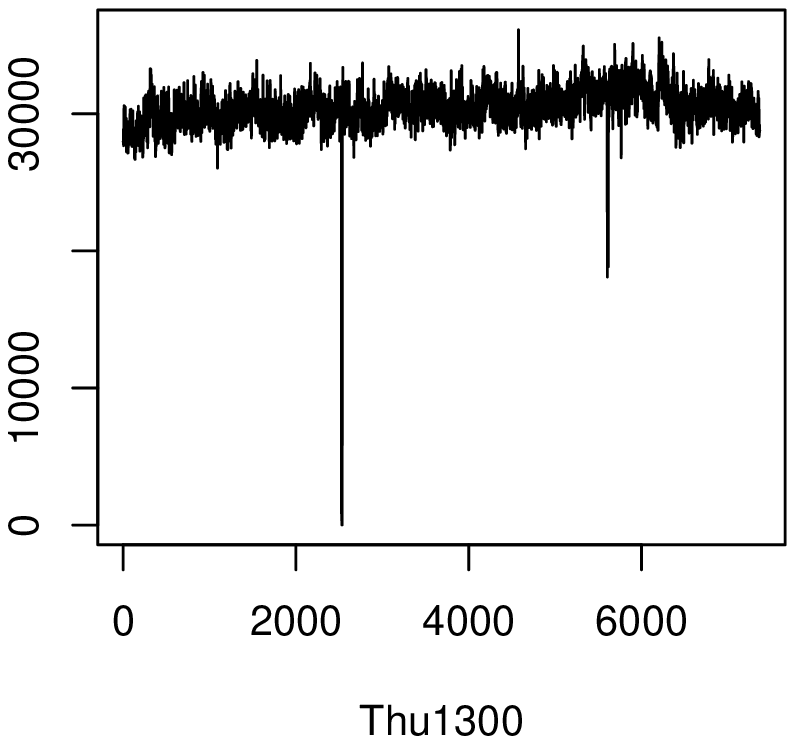}
\end{tabular}
\vspace{-0.7cm}
\caption{\footnotesize{Packet counts of aggregated traffic every 1 second.}}
\label{fig:sat1930}
\end{figure}
The maximal available scale for the two datasets is 20. Since we have
less than 4 observations at this scale, we set the coarse scale
$J_0+\ell=19$ and vary the finest scale $J_0$ from 1 to 17. The values
of the three estimators of $d$ are stored in
Table~\ref{tab:d:cl:rob:mad} for $J_0=1$ to 14 as well as the
standard errors of $\sqrt{n2^{-J_0}}(\hat{d}_{n,\ast}-d)$ for the two datasets:
Thu1300 and Sat1930.

{\tiny
\begin{table}[!h]
\begin{tabular}{ccccccccccccccc}\hline
$J_0$ & 1& 2 & 3 & 4&5&6&7&8&9&10&11&12&13&14\\
\hline\hline
&&&&& \multicolumn{3}{c}{Thu1300} &&&&& \\
\hline
 $\hat{d}_{n,\cl}$&0.08& 0.09 &0.11 &0.15 &0.19 &0.25 &0.31 &0.39& 0.43 &0.47 &0.51 &0.49& 0.44& 0.41\\
 $\mathrm{SE}_\cl$& (0.52) & (0.56) &(0.51) & (0.52) & (0.57)& (0.52)& (0.56)& (1.45)& (0.74)& (0.76)&(0.87) &(0.91) &( 1.10)& (1.21)\\
 $\hat{d}_{n,\CR}$&0.08& 0.07 &0.07 &0.09& 0.13& 0.19 &0.28 & 0.34 &0.37& 0.40& 0.42 &0.43 &0.48 & 0.45\\
$\mathrm{SE}_\CR$ &(0.55) &(0.58)&(0.61) &(0.63) &(0.59) & (0.6)& (0.67) & (1.42) &  (0.82)& (0.88)& (0.97) &(1.08) & (1.18) &(1.23)\\
$\hat{d}_{n,\MAD}$& 0.08 & 0.08& 0.07 & 0.09&0.13& 0.19& 0.27&0.33 &0.38 &0.40 &0.43 & 0.43 &0.5& 0.48\\
$\mathrm{SE}_\MAD$ & (0.74)&(0.87) &(0.78) &(0.83) & (0.86) & (0.84) & (0.91) & (1.49)& (0.98)& (1.04) &(1.07) & (1.15) &(1.18) &(1.2)\\
 \hline\hline
&&&&& \multicolumn{3}{c}{Sat1930}&&&&& \\\hline
$\hat{d}_{n,\cl}$&0.05& 0.06& 0.08& 0.11& 0.14& 0.17& 0.23&0.28&  0.33&  0.36& 0.37 & 0.39& 0.42& 0.42\\
$\mathrm{SE}_\cl$&(0.41) &(0.47)& (0.43)& (0.48)& (0.47)& (0.48)& (0.46)& (0.89)& (0.54)& (0.61)& (0.70)& (0.80)& (1.11) &(1.24)\\
 $\hat{d}_{n,\CR}$& 0.06 &0.06 & 0.06 &0.09 &0.12 &0.16& 0.23& 0.3& 0.34& 0.38& 0.4& 0.42& 0.44& 0.42\\
$\mathrm{SE}_\CR$&(0.51) &(0.47)& (0.54) &(0.48)& (0.48)& (0.53)& (0.56)  &(0.90) & (0.81) &(0.70) &(0.88)& (0.96) &(1.21)& (1.26)\\
$\hat{d}_{n,\MAD}$& 0.06& 0.06& 0.07& 0.09& 0.11& 0.16& 0.23& 0.29& 0.33& 0.38& 0.4& 0.43& 0.45& 0.4\\
$\mathrm{SE}_\MAD$&(0.59)& (0.77)& (0.72)& (0.81)& (0.70)& (0.89)& (0.82)& (0.64)&  (1.13)& (0.99) &(1.10)& (1.34)& (1.49) &(1.38)\\
\hline
 \end{tabular}
\caption{{\footnotesize Estimators of $d$ with $J_0=1$ to $J_0=14$ and
    $J_0+\ell=19$ obtained from Thu1300 and Sat1930. Here SE denotes
    the standard error of $\sqrt{n2^{-J_0}}(\hat{d}_{n,\ast}-d)$.}}
\label{tab:d:cl:rob:mad}
\end{table}
}


In Figure~\ref{fig:estimddata}, we display the estimates $\hat{d}_{n,\cl}$,
$\hat{d}_{n,\CR}$ and $\hat{d}_{n,\MAD}$ of the memory parameter $d$
as well as their respective 95$\%$ confidence intervals from $J_0=1$
to $J_0=14$. We propose to choose $J_0=9$ for Thu1300 and $J_0=10$
for Sat1930 since from these values of $J_0$ the successive confidence
intervals are such that the smallest
one is included in the largest one (for the robust estimators). Note that \cite{park:park:2009}
chose the same values of $J_0$ using another methodology.
For these values of $J_0$ we obtain $\hat{d}_{n,\cl}=0.43$ (with 95\% confidence interval [0.412, 0.443]) ,
$\hat{d}_{n,\CR}=0.37$ (with 95\% confidence interval [0.358, 0.385])  and $\hat{d}_{n,\MAD}=0.38$ with (95\% confidence interval [0.362, 0.397])
for Thu1300 and $\hat{d}_{n,\cl}=0.36$ (with 95\% confidence interval [0.345, 0.374]), $\hat{d}_{n,\CR}=\hat{d}_{n,\MAD}=0.38$ (with 95\% confidence intervals [0.361, 0.398] for $\CR$ and [0.357, 0.402] for $\MAD$) for
Sat1930. These values are similar to the one found by
\cite{park:park:2009}.
\begin{figure}[!h]
\begin{tabular}{cc}
\includegraphics*[width=0.45\textwidth,angle=-90]{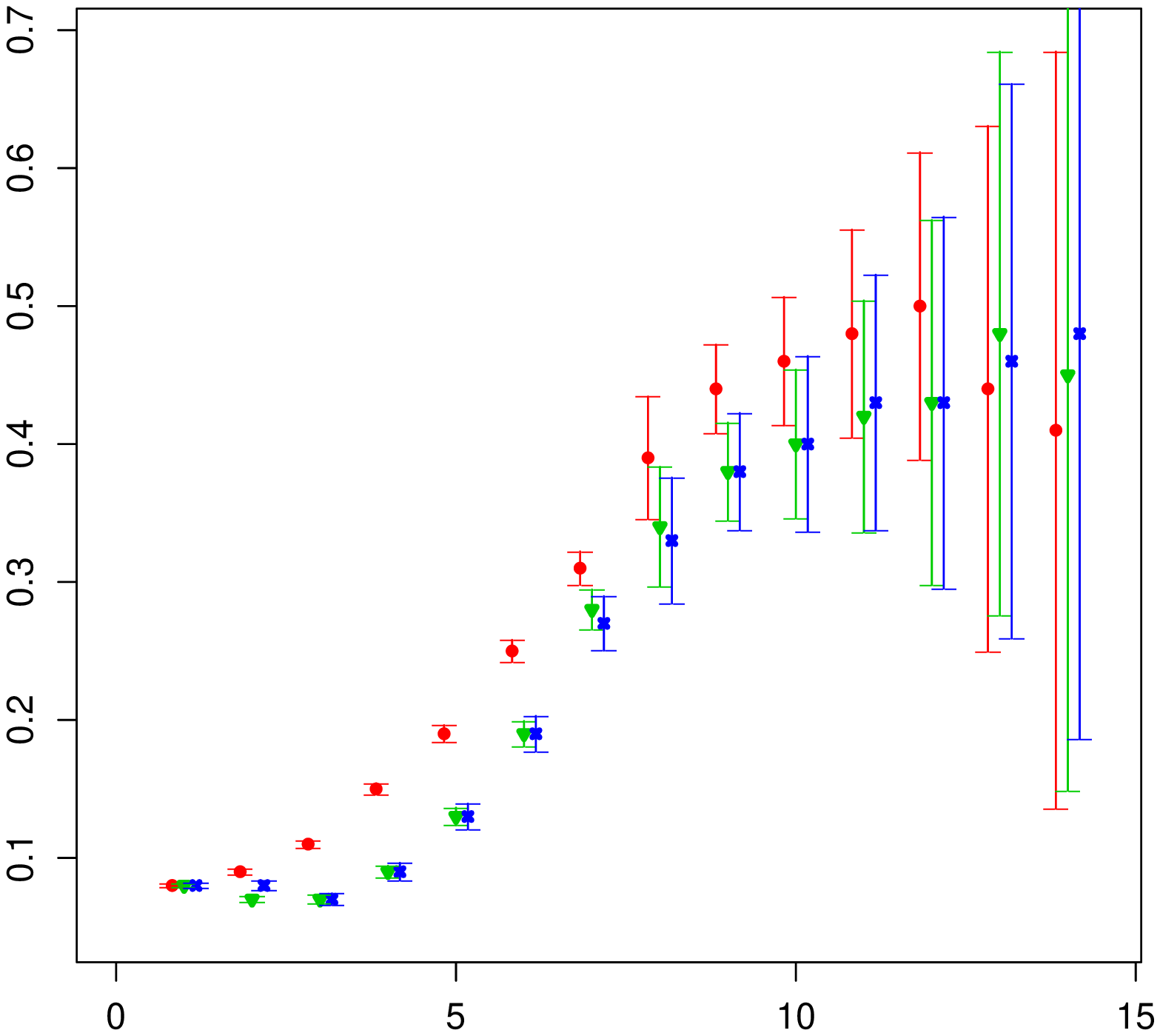}
\includegraphics*[width=0.45\textwidth,angle=-90]{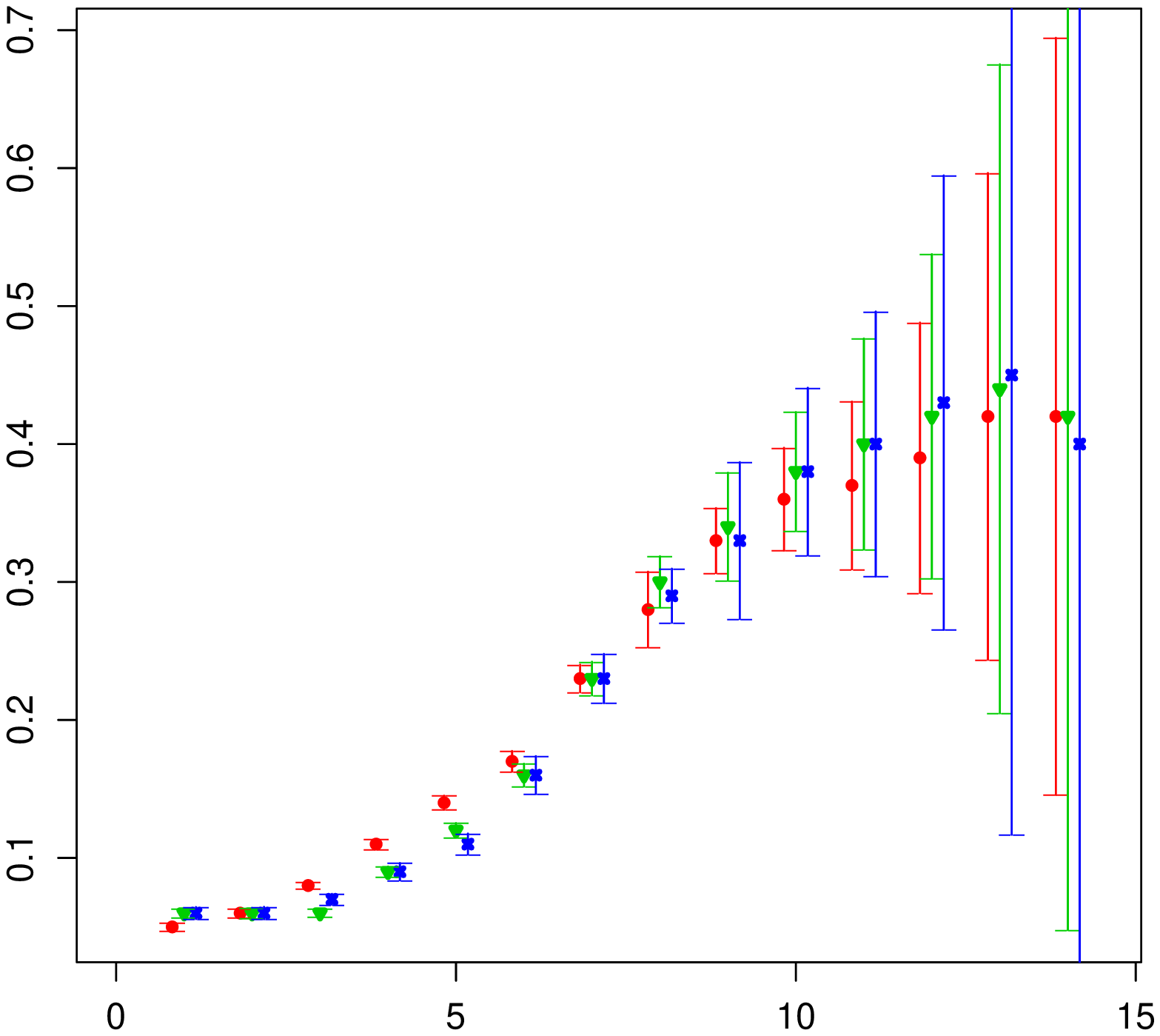}
\end{tabular}
\vspace{-1cm}
\caption{\footnotesize{Confidence intervals of the estimates
    $\hat{d}_{n,\cl}$ (red), $\hat{d}_{n,\CR}$ (green) and
    $\hat{d}_{n,\MAD}$ (blue) on the data Thu1300 (left) and Sat1930 (right) for $J_0=1,\dots,14$
and $J_0+\ell=19$.}}
\label{fig:estimddata}
\end{figure}
%

With this choice of $J_0$ for Thu1300, we observe a significant
difference between the classical estimator and the robust
estimators. Thus to better understand the influence of outliers on the
estimated memory parameter a new dataset with artificial outliers was
generated. The Thu1300 time series shows two spikes shooting
down. Especially, the first downward spike hits
zero. \cite{park:taqqu:stoev:2007} have shown that this dropout lasted
8 seconds. Outliers are introduced by dividing  by 6 the 8000 observations
in this period. The new memory parameter estimators are
$\hat{d}_{n,\cl}=0.445$, $\hat{d}_{n,\CR}=0.375$ and
$\hat{d}_{n,\MAD}=0.377$. As for the Nile River data, the classical
estimator was affected while the robust estimators remain stable.

\section{Proofs}\label{sec:proofs}
Theorem \ref{theo:ext:arcones} is an extension of \cite[Theorem~4]{arcones:1994} to
arrays of stationary Gaussian processes in the unidimensional
case and Theorem \ref{theo:ext:csorgo}
extends the result of \cite{csorgo:mielniczuk:1996}
to arrays of stationary Gaussian processes.
These two theorems are useful for the proof of Proposition~\ref{lemma:asymp:exp}.
\begin{theorem}\label{theo:ext:arcones}
Let $\{X_{j,i},\; j\geq 1,i\geq 0\}$ be an array of
standard stationary Gaussian processes such that for a fixed $j\geq
1$, $(X_{j,i})_{i\geq 0}$
has a spectral density $f_{j}$ and an autocorrelation function $\rho_{j}$
defined by $\rho_{j}(k)=\PE(X_{j,0}X_{j,k})$, for all $k\geq 0$.
Assume also that there exists a non increasing sequence $\{u_{j}\}_{j\geq 1}$ 
such that for all $j\geq1$
\begin{equation}
\label{eq:asumption:fkn}
\sup_{\lambda\in(-\pi,\pi)}|f_{j}(\lambda)-g_\infty(\lambda)|\leq u_{j}\;,
\end{equation}
where $g_\infty$ is a $2\pi$-periodic function which is bounded on $(-\pi,\pi)$ and
continuous at the origin. Let $h$ be a function on
$\Rset$ with Hermite rank $\tau\geq 1$. We assume that $h$ is either
bounded or is a finite linear combination of Hermite polynomials.
Let $\{n_j\}_{j\geq1},$ be a sequence of integers such that $n_j$ tends to infinity as $j$ tends to infinity. Then,
\begin{equation}
\label{eq:clt:ext:arcones}
\frac{1}{\sqrt{n_j}}\sum\limits_{i=1}^{n_j}h\left(X_{j,i}\right)\cd
\mathcal{N}\left(0,\tilde{\sigma}^2\right)\;,\textrm{ as } j\to\infty\; ,
\end{equation}
where
$$\tilde{\sigma}^2=\lim_{n\to\infty}\PVar\Big(\frac{1}{\sqrt{n_j}}\sum_{i=1}^{n_j}h(X_{j,i})\Big)
=\sum_{\ell\geq\tau}\frac{c^2_\ell}{\ell !}g^{\star\ell}_\infty(0).$$
In the previous equality,
$c_\ell=\PE[h(X)H_\ell(X)]$, where $H_\ell$ is the $\ell$-th Hermite
polynomial and $X$ is a standard Gaussian random variable.
\end{theorem}
\begin {proof}[Proof of Theorem~\ref{theo:ext:arcones}]  Let us first prove that
\begin{equation}
\label{eq:arc:hermite}
\frac{\sum_{i=1}^{n_j}\sum_{l\geq\tau}\frac{c_l}{l!}H_l(X_{j,i})}{\sqrt{\PVar\left(\sum_{i=1}^{n_j}\sum_{l\geq\tau}\frac{c_l}{l!}H_l(X_{j,i})\right)}}\cd \mathcal{N}(0,1)\;,\textrm{ as } n\to\infty\;.
\end{equation}
Using Mehler's formula, see Eq. (2.1) of \cite{breuer:major:1983}, we have
\begin{align*}
\PVar\left(\sum_{i=1}^{n_j}\sum_{l\geq\tau}\frac{c_l}{l!}H_l(X_{j,i})\right)&=\sum_{i_1,i_2=1}^{n_j}\sum_{l_1,l_2\geq\tau}\frac{c_{l_1}c_{l_2}}{l_{1}!l_{2}!}\PE\left[H_{l_1}(X_{j,i_1})H_{l_2}(X_{j,i_2})\right]\\
&=\sum_{l\geq\tau}\frac{c_l^2}{l!}\left[\sum_{i_1,i_2=1}^{n_j}\rho_j^l(i_2-i_1)\right].
\end{align*}
 In order to prove \eqref{eq:arc:hermite},
it is enough to prove that for $p\geq1$,
\begin{align}\label{eq:odd}
\frac{\PE\left[\left(\sum_{i=1}^{n_j}\sum_{l\geq\tau}\frac{c_l}{l!}H_l(X_{j,i})\right)^{2p+1}\right]}{\left(\sum_{l\geq\tau}\frac{c_l^2}{l!}\left[\sum_{i_1,i_2=1}^{n_j}\rho_j^l(i_2-i_1)\right]\right)^{\frac{2p+1}{2}}}&\to0,
\text{ as } n\to\infty \text{ and}\\ \label{eq:even}
\frac{\PE\left[\left(\sum_{i=1}^{n_j}\sum_{l\geq\tau}\frac{c_l}{l!}H_l(X_{j,i})\right)^{2p}\right]}{\left(\sum_{l\geq\tau}\frac{c_l^2}{l!}\left[\sum_{i_1,i_2=1}^{n_j}\rho_j^l(i_2-i_1)\right]\right)^{p}}&\to\frac{(2p)!}{p!\,2^p},\text{
  as } n\to\infty.
\end{align}
For all $m\in\Nset^\ast,$
\begin{align*}
\PE\left[\left(\sum_{i=1}^{n_j}\sum_{l\geq\tau}\frac{c_l}{l!}H_l(X_{j,i})\right)^{m}\right]=\sum_{1\leq i_1,\dots,i_m\leq n_j}\sum_{l_1,\dots,l_m\geq \tau}\frac{c_{l_1}\dots c_{l_m}}{l_1!\dots l_m!}\PE\left[H_{l_1}(X_{j,i_1}),\dots,H_{l_m}(X_{j,i_m})\right]\,.
\end{align*}

\noindent
1) We start with the case where $m=2p+1.$

a) Let us first assume that
$|\{i_1,\dots,i_{2p+1}\}|=2p+1$ and that
\begin{equation}\label{eq:cond_rho}
\forall i,\; \rho_j(i)\leq\rho^\ast<1/(2p)\; .
\end{equation}
By \cite[Lemma~3.2 P. 210]{taqqu:1977},
$\PE\left[H_{l_1}(X_{j,i_1}),\dots,H_{l_m}(X_{j,i_m})\right]$ is zero
if $l_1+\dots+l_m$ is odd. Otherwise it is bounded by a constant times
a sum of products of  $(l_1+\dots+l_m)/2$ correlations. Bounding, in
each product, all of them but $p+1$, by $\rho^\ast<1/(2p)$, we get
that
$\PE\left[H_{l_1}(X_{j,i_1}),\dots,H_{l_{2p+1}}(X_{j,i_{2p+1}})\right]$
is bounded by a finite number of terms of the following form
$$
(\rho^\ast)^{\frac{l_1+\dots+l_{2p+1}}{2}-(p+1)}
\rho_j(i_2-i_1)\rho_j(i_4-i_3)\dots \rho_j(i_{2p}-i_{2p-1})\rho_j(i_{2p+1}-i_{2p})
\left|\PE\left(H_{l_1}(X)\dots H_{l_{2p+1}}(X)\right)\right|\; ,
$$
where $X$ is a standard Gaussian random variable.
Note also that the hypercontractivity \cite[Lemma~3.1 P.210]{taqqu:1977} yields
$$\left|\PE\left[H_{l_1}(X)\dots H_{l_{2p+1}}(X)\right]\right|\leq
(2p)^{\frac{l_1+\dots+l_{2p+1}}{2}}\sqrt{l_1!\dots l_{2p+1}!}\;.$$
Thus, using the Cauchy-Schwarz inequality and that $\rho^\ast<\frac{1}{2p}$, there exists a positive
constant $C$ such that
\begin{align*}
&\sum_{l_1,\dots,l_{2p+1}\geq \tau}\frac{|c_{l_1}\dots c_{l_{2p+1}}|}{l_1!\dots l_{2p+1}!}(\rho^\ast)^{\frac{l_1+\dots+l_{2p+1}}{2}-(p+1)} \left|\PE\left(H_{l_1}(X)\dots H_{l_{2p+1}}(X)\right)\right|\\
&\leq \sum_{l_1,\dots,l_{2p+1}\geq \tau}\frac{|c_{l_1}|\dots |c_{l_{2p+1}}|}{\sqrt{l_1!\dots l_{2p+1}!}}(2p\rho^\ast)^{\frac{l_1+\dots+l_{2p+1}}{2}-(p+1)}
\leq (2p\rho^\ast)^{-1}\left(\sum_{l\geq\tau}\frac{|c_l|}{\sqrt{l!}}\left[(2p\rho^\ast)\right]^{\frac{l}{2}-\frac{p}{2p+1}}\right)^{2p+1}\\
&\leq C\left(\sum_{l\geq\tau}\frac{c_l^2}{l!}\right)^{\frac{2p+1}{2}}\left(\sum_{l\geq\tau}(2p\rho^\ast)^{l-\frac{2p}{2p+1}}\right)^\frac{2p+1}{2}<\infty\;.
\end{align*}
To conclude the proof of \eqref{eq:odd}, it remains to prove that
\begin{equation}\label{eq:clt:odd}
\frac{\sum\limits_{\substack{1\leq i_1,\dots, i_{2p+1}\leq n_j\\
      |\{i_1,\dots,i_{2p+1}\}|=2p+1}}\rho_j(i_2-i_1)\rho_j(i_4-i_3)\dots \rho_j(i_{2p}-i_{2p-1})\rho_j(i_{2p+1}-i_{2p})}{\left(\sum\limits_{l\geq\tau}\frac{c_l^2}{l!}\left[\sum\limits_{i_1,i_2=1}^{n_j}\rho_j^l(i_2-i_1)\right]\right)^{p+\frac{1}{2}}}\to 0,\text{ as }n_j\to \infty\;.
\end{equation}
Let us first study the numerator in the l.h.s of \eqref{eq:clt:odd}.
\begin{multline*}
\sum\limits_{\substack{1\leq i_1,\dots i_{2p+1}\leq n_j\\ |\{i_1,\dots,i_{2p+1}\}|=2p+1}}\rho_j(i_2-i_1)\rho_j(i_4-i_3)\dots \rho_j(i_{2p}-i_{2p-1})\rho_j(i_{2p+1}-i_{2p})\\
=\big(\sum_{1\leq i_1\neq i_2\leq n_j}\rho_j(i_2-i_1)\big)^{p-1}\sum_{\substack{1\leq i_{2p-1},i_{2p},i_{2p+1}\leq n_j\\ |\{i_{2p-1},i_{2p},i_{2p+1}\}|=3}}\rho_j(i_{2p}-i_{2p-1})\rho_j(i_{2p+1}-i_{2p})\\
=\big(\sum_{1\leq i_1\neq i_2\leq n_j}\rho_j(i_2-i_1)\big)^{p-1}\sum_{i_{2p}=1}^{n_j}\big(\sum_{1\leq i_{2p}\neq i_{2p+1}\leq n_j}\rho_j(i_{2p+1}-i_{2p})\big)^2\;.
\end{multline*}
To prove \eqref{eq:clt:odd}, we start by proving that
\begin{equation}\label{eq:rho:rneqs}
\sum_{r=1}^{n_j}\Big(\sum_{{1\leq s\leq n_j}}\rho_j(r-s)\Big)^2=O(n_j).
\end{equation}
Using the notation $D_{n_j}(\lambda)=\sum_{r=1}^{n_j}\rme^{\rmi\lambda r}$, we get
\begin{align*}
\sum_{r=1}^{n_j}\Big(\sum_{{1\leq s\leq n_j}}\rho_j(r-s)\Big)^2&=\sum_{r=1}^{n_j}\Big(\int_{-\pi}^{\pi}\rme^{\rmi\lambda r}\sum_{{1\leq s\leq n_j}}\rme^{-\rmi\lambda s}f_j(\lambda)\d\lambda\Big)^2\\
&=\int_{-\pi}^{\pi}\int_{-\pi}^{\pi}D_{n_j}(\lambda-\lambda')D_{n_j}(\lambda)\overline{D_{n_j}(\lambda')}
f_j(\lambda)f_j(\lambda')\d\lambda\d\lambda'\;.
\end{align*}
Using \eqref{eq:asumption:fkn}, the boundedness of $g_\infty$ and that
$u_j$ is bounded, there exists a positive constant $C$ such that
\begin{multline*}
|f_j(\lambda)f_j(\lambda')|\leq |f_j(\lambda)-g_\infty(\lambda)| |f_j(\lambda')-g_\infty(\lambda')|+|g_\infty(\lambda')| |f_j(\lambda)-g_\infty(\lambda)|\\
+|g_\infty(\lambda)| |f_j(\lambda')-g_\infty(\lambda')|+ |g_\infty(\lambda)||g_\infty(\lambda')|
\leq C\;.
\end{multline*}
Then, using that there exists a positive constant $c$ such that
$|D_{n_j}(\lambda)|\leq {c n_j}/(1+n_j|\lambda|)$, for all $\lambda$
in $[-\pi,\pi]$,
\begin{align}\label{eq:integral:dn}
\sum_{r=1}^{n_j}\Big(\sum_{{1\leq s\leq n_j}}\rho_j(r-s)\Big)^2\leq
c^3n_j\int_{\Rset^2}\frac{1}{1+|\mu-\mu'|}\frac{1}{1+|\mu|}\frac{1}{1+|\mu'|}\d\mu\d\mu'\; .
\end{align}
The result \eqref{eq:rho:rneqs} thus follows from the convergence of
the integral in \eqref{eq:integral:dn} which is proved in Lemma
\ref{lem:conv:integrale}.
Let us now prove that
\begin{equation}\label{eq:lim:rho}
\frac{1}{n_j}\sum_{1\leq r, s\leq n_j}\rho_j(r-s)\to g_{\infty}(0)\;,
\textrm{ as } n\to\infty\; .
\end{equation}
Using that $F_j$ defined by $F_j(\lambda)=(2\pi
n_j)^{-1}|\sum\limits_{r=1}^{n_j}\rme^{\rmi\lambda r}|^2$,
for all $\lambda$ in $[-\pi,\pi]$ satisfies $\int_{-\pi}^{\pi}F_j(\lambda)\d\lambda=1,$ we obtain
\begin{multline}\label{eq:dec:rho}
\frac{1}{n_j}\Big(\sum_{1\leq r, s\leq
  n_j}\rho_j(r-s)\Big)-g_\infty(0)=\int_{-\pi}^{\pi}\left(f_j(\lambda)-g_\infty(\lambda)\right)F_j(\lambda)\d\lambda
+\int_{-\pi}^{\pi}\left(g_\infty(\lambda)-g_\infty(0)\right)F_j(\lambda)\d\lambda\,.
\end{multline}
Using that  $\int_{-\pi}^{\pi}F_j(\lambda)\d\lambda=1$ and
(\ref{eq:asumption:fkn}), the first term in the r.h.s of \eqref{eq:dec:rho} tends to zero as $n$
tends to infinity.
The second term in the r.h.s of \eqref{eq:dec:rho} can be upper
bounded as follows. For $0<\eta\leq\pi$,
\begin{multline}\label{eq:dec:terme2}
\left|\int_{-\pi}^{\pi}\left(g_\infty(\lambda)-g_\infty(0)\right)F_j(\lambda)\d\lambda\right|\leq \int_{-\pi}^{-\eta}|g_\infty(\lambda)-g_\infty(0)|F_j(\lambda)\d\lambda\\+\int_{-\eta}^{\eta}|g_\infty(\lambda)-g_\infty(0)|F_j(\lambda)\d\lambda+\int_{\eta}^{\pi}|g_\infty(\lambda)-g_\infty(0)|F_j(\lambda)\d\lambda\,.
\end{multline}
Since there exists a positive constant $C$ such that $F_j(\lambda)\leq C/(n_j|\lambda|^2)$,
for all $\lambda$ in $[-\pi,\pi]$, the first and last terms in the
r.h.s of
\eqref{eq:dec:terme2} are bounded by $C\pi/(n_j \eta^2)$. The continuity of $g_\infty$ at $0$
and the fact that
$\int_{-\eta}^{\eta}F_j(\lambda)\d\lambda\leq\int_{-\pi}^{\pi}F_j(\lambda)\d\lambda=1$
ensure that the second term in the r.h.s of \eqref{eq:dec:terme2} tends to zero as $n$
tends to infinity. This concludes the proof of \eqref{eq:lim:rho}.

Using the same arguments as those used to prove \eqref{eq:lim:rho}
and the fact that $\rho^l_j$ is  the autocorrelation associated to
$f_j^{\star l}$ which is the $l$-th self-convolution of $f_j$, we get
that
\begin{equation}\label{eq:rho:conv}
\frac{1}{n_j}\sum_{r,s=1}^{n_j}\rho^l_j(r-s)\to g^{\star l}_\infty(0), \text{ as } n\to \infty\, .
\end{equation}

Let us now prove that the denominator in \eqref{eq:clt:odd} is
$O(n_j^{p+\frac{1}{2}})$ as $n\to\infty.$
We aim at applying Lemma~\ref{lemma:double:fatou} with $f_n,$ $g_n$,
$f$ and $g$ defined hereafter.
$$f_{n_j}(s,l)=\frac{c_l^2}{l!}\1_{\{|s|<n_j\}}\Big(1-\frac{|s|}{n_j}\Big)\rho_j^l(s).$$
Observe that $|f_{n_j}(s,l)|\leq g_{n_j}(s,l)$ where
$$g_{n_j}(s,l)=\frac{c_l^2}{l!}\1_{\{|s|<n_j\}}\Big(1-\frac{|s|}{n_j}\Big)\rho_j^2(s).$$
Using \eqref{eq:asumption:fkn} and the fact that the spectral density
associated to $\rho_j^l$ is $f_j^{\star l}$,
we get, as $n\to \infty$,
$$
f_{n_j}(s,l)\to f(s,l)=\frac{c_l^2}{l!}\int_{-\pi}^{\pi}g_\infty^{\star l}(\lambda)\rme^{\rmi\lambda s}\d\lambda \textrm{ and } g_{n_j}(s,l)\to g(s,l)=\frac{c_l^2}{l!}\int_{-\pi}^{\pi}g_\infty^{\star 2}(\lambda)\rme^{\rmi\lambda s}\d\lambda\;.
$$
Using \cite[Lemma 1]{moulines:roueff:taqqu:2007:jtsa}, we get
$$
\sum_{l\geq\tau}\sum_{s\in\Zset}g_{n_j}(s,l)\to \sum_{l\geq\tau}\frac{c_l^2}{l!}g_\infty^{\star 2}(0)\;.
$$
Then, Lemma~\ref{lemma:double:fatou} yields
$$\lim_{n\to\infty}\frac{1}{n_j}\PVar\Big(\sum_{i=1}^{n_j}\sum_{l\geq\tau}\frac{c_l}{l!}H_l(X_{j,i})\Big)
=\lim_{n\to\infty}\frac{1}{n_j}\sum_{l\geq\tau}\frac{c_l^2}{l!}\Big[\sum_{i_1,i_2=1}^{n_j}\rho_j^l(i_2-i_1)\Big]
=\sum_{l\geq\tau}\frac{c_l^2}{l!}g_\infty^{\star l}(0)\;.$$
Hence we get \eqref{eq:clt:odd} by noticing that the numerator in
\eqref{eq:clt:odd} is $O(n^p_j)$.

If Condition (\ref{eq:cond_rho}) is not satisfied then let $k_0$ be
such that $\rho_j(k)\leq\rho^\ast<1/(2p)$, for all $k>k_0$. In the
case where $h$ is a linear combination of $L$ Hermite polynomials,
the same arguments as those used previously are valid with
$\rho^\ast=1$.
In the case where $h$ is bounded, there exists a positive constant $C$
such that
\begin{equation}\label{eq:maj:h_bound}
\PE\bigg[\bigg(\sum_{i=1}^{n_j}h(X_{j,i})\bigg)^{2p+1}\bigg]
\leq C \sum_{1\leq i_1,\dots,i_q\leq n_j}
\PE\big[|h|(X_{j,i_1})\dots |h|(X_{j,i_q})\big]\;,
\end{equation}
where $i_1,\dots,i_q$ are such that $|i_k-i_l|>k_0$, for all
$k,l$ in $\{1,\dots,q\}$ with $q\leq 2p+1$. By expanding $|h|$
onto the basis of Hermite polynomials, we can conclude with the
same arguments as those used when Condition (\ref{eq:cond_rho})
is valid.

b) Let us now assume that $|\{i_1,\dots,i_{2p+1}\}|=r\leq 2p$.
In the case where $h$ is bounded, the inequality
(\ref{eq:maj:h_bound}) is valid with $q\leq r$ which gives
that the numerator of (\ref{eq:odd}) is $O(n_j^{\pent{r/2}}).$
In the
case where $h$ is a linear combination of $L$ Hermite polynomials,
we use the same arguments as those used in a)
with
$\rho^\ast=1$ which implies that
the numerator of (\ref{eq:odd}) is $O(n_j^{\pent{r/2}}).$

2) Let us now study the case where $m$ is even that is $m=2p$ with $p\geq1.$
\begin{align}\label{eq:prod_hermite_even}
\PE\bigg[\Big(\sum_{i=1}^{n_j}\sum_{l\geq\tau}\frac{c_l}{l!}H_l(X_{j,i})\Big)^{2p}\bigg]
=\sum_{1\leq i_{1},\dots,i_{2p}\leq n_j}\sum_{ l_1,\dots,l_{2p}\geq
  \tau}
\frac{c_{l_1}\dots c_{l_{2p}}}{l_1!\dots
  l_{2p}!}\PE\left[H_{l_1}(X_{j,i_1})\dots H_{l_{2p}}(X_{j,i_{2p}})\right]\,.
\end{align}
By \cite[Formula~(33), P.69]{rosenblatt:1985}, we have
\begin{align}\label{eq:rosenblatt}
\PE\left[H_{l_1}(X_{j,i_1})\dots
  H_{l_{2p}}(X_{j,i_{2p}})\right]=l_1!\dots l_{2p}!\sum_{\{l_1,\dots,l_{2p}\}}\frac{\rho_j^\nu}{\nu!}\,,
\end{align}
where it is understood that $\rho_j^\nu=\prod\limits_{1\leq q<k\leq
  2p}\rho_j^{\nu_{q,k}}(q-k)$,
$\nu!=\prod\limits_{1\leq q<k\leq 2p}\nu_{q,k}!$,
and $\sum_{\{l_1\dots,l_{2p}\}}$ indicates that we are to sum
over all symmetric matrices $\nu$ with nonnegative integer
entries, $\nu_{ii}=0$ and the row sums equal to $l_1,\dots,l_{2p}.$

We shall prove that
among all the terms in the r.h.s of \eqref{eq:rosenblatt},
the leading ones correspond to the case where we have $p$
pairs of equal indices in the set $\{l_1,\dots,l_{2p}\}$, that
is, for instance,
$l_1=l_2,\,l_3=l_4,\dots,l_{2p-1}=l_{2p}$ and
$\nu_{1,2}=l_1$, $\nu_{3,4}=l_3$,...,$\nu_{2p-1,2p}=l_{2p-1}$
the others $\nu_{i,j}$ being equal to zero. This gives
\begin{align*}
(l_2!)^2\dots(l_{2p}!)^2
\frac{\rho_j(i_2-i_1)^{l_2}\rho_j(i_4-i_3)^{l_4}\dots
  \rho_j(i_{2p}-i_{2p-1})^{l_{2p}}}{l_2!\dots l_{2p}!}\; .
\end{align*}
The corresponding term in \eqref{eq:prod_hermite_even} is given by
\begin{multline*}
\sum_{1\leq i_{1},\dots,i_{2p}\leq n_j}\sum_{ l_2,l_4,\dots,l_{2p}\geq
  \tau}
\frac{c^2_{l_2} c^2_{l_4}\dots c^2_{l_{2p}}}{l_2! l_4!\dots l_{2p}!}
\rho_j(i_2-i_1)^{l_2}\rho_j(i_4-i_3)^{l_4}\dots\rho_j(i_{2p}-i_{2p-1})^{l_{2p}}\\
=\Big[\sum_{l\geq\tau}\frac{c^2_l}{l!}\Big(\sum_{i_1,i_2=1}^{n_j}\rho^l_k(i_2-i_1)\Big)\Big]^p\;,
\end{multline*}
which corresponds to the denominator in the l.h.s of
\eqref{eq:even}. Since there exists exactly
${(2p)!}/(2^pp!)$ possibilities to have pairs
of equal indices among $2p$ indices we obtain (\ref{eq:even}) if we
prove that the other terms can be neglected.

Let us first consider the case where
\begin{equation}\label{eq:cond_rho_pair}
\forall i, \;\rho_j(i)\leq\rho^\ast<\frac{1}{2p-1}
\end{equation}
and $|\{i_{1},\dots,i_{2p}\}|=2p$.
By \cite[Lemma~3.2 P. 210]{taqqu:1977},
$\PE\left[H_{l_1}(X_{j,i_1}),\dots,H_{l_m}(X_{j,i_m})\right]$ is zero
if $l_1+\dots+l_m$ is odd. Otherwise it is bounded by a constant times
a sum of products of  $(l_1+\dots+l_m)/2$ correlations. Bounding, in
each product, all of them but $p+1$, by $\rho^\ast<1/(2p-1)$, we get
that
$\PE\left[H_{l_1}(X_{j,i_1}),\dots,H_{l_{2p}}(X_{j,i_{2p}})\right]$
is bounded by a finite number of terms of the following form
$$
(\rho^\ast)^{\frac{l_1+\dots+l_{2p}}{2}-(p+1)}
\rho_j(i_2-i_1)\rho_j(i_4-i_3)\dots \rho_j(i_{2p}-i_{2p-1})\rho_j(i_{2p}-i_{1})
\left|\PE\left(H_{l_1}(X)\dots H_{l_{2p}}(X)\right)\right|\; .
$$
where $X$ is a standard Gaussian random variable.
Using the same arguments as in the case where $m$ was odd, we have
\begin{align*}
\sum_{l_1,\dots,l_{2p}\geq \tau}\frac{|c_{l_1}\dots c_{l_{2p}}|}{l_1!\dots l_{2p}!}(\rho^\ast)^{\frac{l_1+\dots+l_{2p}}{2}-(p+1)}\left|\PE\left(H_{l_1}(X)\dots H_{l_{2p}}(X)\right)\right|<\infty\,.
\end{align*}
To have the result \eqref{eq:even}, it remains to show that
\begin{align}\label{eq:even1}
\frac{\sum_{\substack{1\leq i_1,\dots, i_{2p}\leq n_j\\
      |\{i_1,\dots,i_{2p}\}|=2p}}\rho_j(i_2-i_1)\rho_j(i_4-i_3)\dots
  \rho_j(i_{2p}-i_{2p-1})\rho_j(i_{2p}-i_{1})}{\left[\sum_{l\geq\tau}\frac{c^2_l}{l!}\left(\sum_{i_1,i_2=1}^{n_j}\rho^l_j(i_2-i_1)\right)\right]^p}\to0\;, \textrm{ as } n\to \infty\;.
\end{align}
The numerator of \eqref{eq:even1} can be rewritten as
\begin{multline*}
\sum_{\substack{1\leq i_1,\dots, i_{2p}\leq n_j\\
    |\{i_1,\dots,i_{2p}\}|=2p}}
\rho_j(i_2-i_1)\rho_j(i_4-i_3)\dots \rho_j(i_{2p}-i_{2p-1})\rho_j(i_{2p}-i_{1})\\
=\Big(\sum_{1\leq i_3\neq i_4\leq n_j}\rho_j(i_4-i_3)\Big)^{p-2}\bigg[\sum_{\substack{1\leq i_1,i_2,i_{2p-1}, i_{2p}\leq n_j\\ |\{i_1,i_2,i_{2p-1},i_{2p}\}|=4}}\rho_j(i_2-i_1)\rho_j(i_{2p}-i_{2p-1})\rho_j(i_{2p}-i_1)\bigg]\,.
\end{multline*}
Using \eqref{eq:lim:rho}, we have $\left(\sum_{1\leq i_3\neq i_4\leq
    n_j}\rho_j(i_4-i_3)\right)^{p-2}=O(n^{p-2}_j).$
Let us now prove that
\begin{align}\label{eq:rho:rneq:even}
\sum_{1\leq i_1,i_2,i_{3}, i_{4}\leq n_j
}\rho_j(i_2-i_1)\rho_j(i_{3}-i_{4})\rho_j(i_{3}-i_1)=O(n_j)\; .
\end{align}
Using the notation $D_{n_j}(\lambda)=\sum_{r=1}^{n_j}\rme^{\rmi \lambda r}$,
\begin{multline*}
\sum_{1\leq i_1,i_2,i_{3}, i_{4}\leq n_j }\rho_j(i_2-i_1)\rho_j(i_{3}-i_{4})\rho_j(i_{3}-i_1)\\
=\sum_{1\leq i_1,i_2,i_{3}, i_{4}\leq n_j}\Big(\int_{-\pi}^{\pi}\rme^{\rmi\lambda(i_2-i_1)}f_j(\lambda)\d\lambda\Big) \Big(\int_{-\pi}^{\pi}\rme^{\rmi\mu(i_3-i_4)}f_j(\mu)\d\mu\Big)\Big(\int_{-\pi}^{\pi}\rme^{\rmi\xi(i_3-i_1)}f_j(\xi)\d\xi\Big)\;\\
=\int_{-\pi}^{\pi}f_j(\xi)\bigg(\int_{-\pi}^{\pi}\overline{D_{n_j}(\mu)}D_{n_j}(\mu+\xi)f_j(\mu)\d\mu \int_{-\pi}^{\pi}D_{n_j}(\lambda)\overline{D_{n_j}(\lambda+\xi)}f_j(\lambda)\d\lambda\bigg)\d\xi\\
\leq\int_{-\pi}^{\pi}\bigg(\int_{-\pi}^{\pi}|D_{n_j}(\lambda)||D_{n_j}(\lambda+\xi)|f_j(\lambda)
\d\lambda\bigg)^2f_j(\xi)\d\xi\;.
\end{multline*}
Using (\ref{eq:asumption:fkn}) and that $g_\infty$ is bounded,
\eqref{eq:rho:rneq:even} will follow if we prove that
$\int_{-\pi}^{\pi}(\int_{-\pi}^{\pi}|D_{n_j}(\lambda)||D_{n_j}(\lambda+\xi)|
\d\lambda)^2\d\xi =O(n_j)\;.$
Since there exists a positive constant $c$ such that
$|D_{n_j}(\lambda)|\leq {c n_j}/(1+n_j|\lambda|)$, for all $\lambda$
in $[-\pi,\pi]$,
\begin{multline}\label{eq:integrale:even}
\int_{-\pi}^{\pi}\bigg(\int_{-\pi}^{\pi}|D_{n_j}(\lambda)||D_{n_j}(\lambda+\xi)|f_j(\lambda)
\d\lambda\bigg)^2f_j(\xi)\d\xi
\leq c^4n_j\int_{-\infty}^{\infty}\bigg(\int_{-\infty}^{\infty}\frac{1}{1+|\mu|}\frac{1}{|1+\mu+\mu'|}\d\mu\bigg)^2\d\mu'
\end{multline}
The result \eqref{eq:rho:rneq:even} thus follows from the convergence of
the last integral in \eqref{eq:integrale:even} which is proved in Lemma
\ref{lem:conv:integrale:even}. Hence we get (\ref{eq:even1}) since the
numerator of the l.h.s of (\ref{eq:even1}) is $O(n_j^{p-1})$ and the
denominator is $O(n_j^p)$ by the same arguments as those used to
find the order of the denominator of (\ref{eq:clt:odd}).
If Condition \eqref{eq:cond_rho_pair} is not satisfied or if
$|\{i_1,\dots,i_{2p}\}|<2p$, we can use similar arguments as those
used in 1)a) and 1)b) to conclude the proof.
\end{proof}
\begin{theorem}\label{theo:ext:csorgo}
Let $\{X_{j,i},\; j\geq 1,i\geq 0\}$ be an array of
standard stationary Gaussian processes such that for a fixed $j\geq
1$, $(X_{j,i})_{i\geq 0}$
has a spectral density $f_{j}$ and an autocorrelation function $\rho_{j}$
defined by $\rho_{j}(k)=\PE(X_{j,0}X_{j,k})$, for all $k\geq 0$.
Let $F_j$ be the c.d.f of $X_{j,1}$ and $F_{n_j}$ the empirical c.d.f computed from
$X_{j,1},\dots,X_{j,n_j}$.
If Condition (\ref{eq:asumption:fkn}) holds,
\begin{equation}
\label{eq:ext:csorgo}
\sqrt{n_j}(F_{n_j}-F_j)\cd W\quad\text{in}\quad D([-\infty,\infty])\,,
\end{equation}
where $W$ is a Gaussian process and $D([-\infty,\infty])$ denotes the Skorokhod space on $[-\infty,\infty]$.
\end{theorem}
\begin{proof}[Proof of Theorem~\ref{theo:ext:csorgo}]
Let
$S_j(x)=n_j^{-1/2}\sum_{i=1}^{n_j}\left(\1_{\{X_{j,i}\leq
    x\}}-F_j(x)\right)$, for all $x$ in $\mathbb{R}$.
We shall first prove that for $x_1,\dots,x_Q$ and
$a_1,\dots,a_Q$ in $\mathbb{R}$
\begin{equation}
  \label{eq:fidis}
  \sum_{q=1}^Q a_q
  S_j(x_q)\cd\mathcal{N}\left(0,\sum_{l\geq1}\frac{c_l^2}{l!}g_\infty^{\star l}(0)\right)\;,
\textrm{ as } n\to\infty\;,
\end{equation}
where $c_l$ is the $l$-th Hermite coefficient of the function $h$
defined by
$$h(\cdot)=\sum_{q=1}^Q a_q\left(\1_{\{\cdot\leq
    x_q\}}-\PE(\1_{\{\cdot\leq x_q\}})\right)\;.$$
Thus,
$
\sum_{q=1}^Q a_q S_j(x_q)=n_j^{-1/2}\sum_{i=1}^{n_j}h(X_{j,i}),
$
where $h$ is bounded and of Hermite rank $\tau\geq 1$ since for all $t$ in $\Rset$,
$\PE(X\1_{X\leq t})=\int_{\Rset}x\1_{x\leq t}\varphi(x)\d x=\int_{-\infty}^t(-\varphi(x))'\d x=-\varphi(t)\neq0,$ and the CLT \eqref{eq:fidis} follows from Theorem~\ref{theo:ext:arcones}.

Let us now prove that there exists a positive constant $C$ and
$\beta>1$ such that for all $ r\leq s \leq t$,
\begin{equation}\label{eq:tightness:csorgo}
\PE\left(|S_j(s)-S_j(r)|^2|S_j(t)-S_j(s)|^2\right)\leq C |t-r|^{\beta}\;.
\end{equation}
The convergence \eqref{eq:ext:csorgo} then follows from
\eqref{eq:fidis},
\eqref{eq:tightness:csorgo}
and \cite[Theorem 13.5]{billingsley:1999}. Note that
 \begin{multline*}
\PE\left(|S_j(s)-S_j(r)|^2|S_j(t)-S_j(s)|^2\right)\\=\frac{1}{n_j^2}\sum_{i,i'=1}^{n_j}\sum_{l,l'=1}^{n_j}\PE\left((f_s-f_r)(X_{j,i})(f_s-f_r)(X_{j,i'})(f_t-f_s)(X_{j,l})(f_t-f_s)(X_{j,l'}\right)\,,
 \end{multline*}
 where $f_t(X)=\1_{\{X\leq t\}}-\PE(\1_{\{X\leq t\}}).$
By developing each difference of functions in Hermite polynomials , we get
 \begin{multline*}
\PE\left(|S_j(s)-S_j(r)|^2|S_j(t)-S_j(s)|^2\right)=
\frac{1}{n_j^2}\sum_{i,i'=1}^{n_j}\sum_{l,l'=1}^{n_j}\sum_{p_1,\dots,p_4\geq1}\\
\frac{c_{p_1}(f_s-f_r)c_{p_2}(f_s-f_r)c_{p_3}(f_t-f_s)c_{p_4}(f_t-f_s)}{p_1!\dots
  p_4!}\PE\left(H_{p_1}(X_{j,i})H_{p_2}(X_{j,i'})H_{p_3}(X_{j,l})H_{p_4}(X_{j,l'})\right)\;.
 \end{multline*}
Using the same arguments as in the case where $m$ is even in the proof
of Theorem~\ref{theo:ext:arcones},
we obtain
\begin{multline*}
\PE\left(|S_j(s)-S_j(r)|^2|S_j(t)-S_j(s)|^2\right)=\frac{1}{n_j^2}\sum_{p_1,p_2\geq1}\sum_{i,i',l,l'=1}^{n_j}\Big[\frac{c^2_{p_1}(f_t-f_s)c^2_{p_2}(f_s-f_r)}{p_1!p_2!}\\\rho_j^{p_1}(i'-i)\rho_j^{p_2}(l'-l)+\frac{c_{p_1}(f_t-f_s)c_{p_1}(f_s-f_r)c_{p_2}(f_t-f_s)c_{p_2}(f_s-f_r)}{p_1!p_2!}\rho_j^{p_1}(l-i)\rho_j^{p_2}(l'-i')\\+\frac{c_{p_1}(f_t-f_s)c_{p_1}(f_s-f_r)c_{p_2}(f_t-f_s)c_{p_2}(f_s-f_r)}{p_1!p_2!}\rho_j^{p_1}(l'-i)\rho_j^{p_2}(l-i')\Big]+O(n_j^{-1})\;.
\end{multline*}
Let $\|\cdot\|_2=(\PE(\cdot)^2)^{1/2}$
and $\pscal{f}{g}=\PE[f(X)g(X)]$, where $X$ is a standard Gaussian
random variable.
Since, by (\ref{eq:rho:conv}),
$\sum_{i,i',l,l'=1}^{n_j}\rho_j^{p_1}(l-i)\rho_j^{p_2}(l'-i')=O(n_j^2)$,
we get with the Cauchy-Schwarz inequality that there exists a positive
constant $C$ such that
\begin{multline*}
\PE\left(|S_j(s)-S_j(r)|^2|S_j(t)-S_j(s)|^2\right)\\\leq C
\sum_{p_1,p_2\geq1}\Big[\frac{c^2_{p_1}(f_t-f_s)c^2_{p_2}(f_s-f_r)}{p_1!p_2!}
+\frac{c_{p_1}(f_t-f_s)c_{p_1}(f_s-f_r)c_{p_2}(f_t-f_s)c_{p_2}(f_s-f_r)}{p_1!p_2!}\Big]
\\\leq
C\left(\left\|f_t-f_s\right\|^2_2\left\|f_s-f_r\right\|^2_2+\left|\pscal{f_t-f_s}{f_s-f_r}\right|^2\right)
\leq C\left\|f_t-f_s\right\|^2_2\left\|f_s-f_r\right\|^2_2\;.
\end{multline*}
Note that
$\left\|f_t-f_s\right\|^2_2
\leq 2 \big(\left\|\1_{\{X\leq t\}}-\1_{\{X\leq s\}}\right\|_2^2+\left\|\PE(\1_{\{X\leq s\}})-\PE(\1_{\{X\leq t\}})\right\|_2^2\big).$
Since $s\leq t$, $\left\|\1_{\{X\leq t\}}-\1_{\{X\leq
    s\}}\right\|_2^2=\Phi(t)-\Phi(s)\leq C|t-s|$, where $\Phi$ denotes
the c.d.f of a standard Gaussian random variable. Moreover,
$\left\|\PE(\1_{\{X\leq s\}})-\PE(\1_{\{X\leq t\}})\right\|_2^2\leq C
|t-s|^2,$
which concludes the proof of (\ref{eq:tightness:csorgo}).
\end{proof}
\begin{proof} [Proof of Proposition~\ref{lemma:asymp:exp}]
 We first prove \eqref{exp} for $\ast=\cl.$
 \begin{align*}
\sqrt{n_j}(\hat{\sigma}^2_{\cl,j}-\sigma^2_j)&=\frac{1}{\sqrt{n_j}}\sum_{i=0}^{n_j-1}(W^2_{j,i}-\sigma_j^2)
=\frac{2\sigma_j^2}{\sqrt{n_j}}\sum_{i=0}^{n_j-1}\frac{1}{2}\bigg(\frac{W^2_{j,i}}{\sigma^2_j}-1\bigg)\;.
 \end{align*}
Let us now prove \eqref{exp} for $\ast=\MAD.$
Let us denote by $F_{n_j}$ the empirical c.d.f of $W_{j,0:n_j-1}$ and
by $F_j$ the c.d.f
of $W_{j,0}$.
Note that
$$\hat{\sigma}_{\MAD,j}=m(\Phi)T_0(F_{n_j})\; ,$$
where $T_0=T_2 \circ T_1$ with
$T_1: F\mapsto \left\{r \mapsto \int_{\Rset}  \1_{\{\vert x\vert\leq
    r\}}\rmd F(x)\right\}$ and
$T_2:U\mapsto U^{-1}(1/2)$. To prove \eqref{exp}, we start by proving that $\sqrt{n_j}(F_{n_j}-F_j)$ converges in distribution in the space of cadlag functions equipped with the topology of uniform convergence. This convergence follows by applying Theorem~\ref{theo:ext:csorgo} to $X_{j,i}=W_{j,i}/\sigma_j$
which is an array of zero mean stationary Gaussian processes by
\cite[Corollary 1]{moulines:roueff:taqqu:2007:jtsa}.
The spectral density $f_j$ of $(X_{j,i})_{i\geq 0}$ is given by $f_j(\lambda)
=\bdens[\phi,\psi]{j,0}{\lambda}{f}/\sigma_j^2$ where
$\bdens[\phi,\psi]{j,0}{\cdot}{f}$
is the within scale spectral density of the process
$\{W_{j,k}\}_{k\geq0}$
defined in \eqref{eq:def:cov:rob} and $\sigma^2_j$ is the wavelet spectrum defined in \eqref{eq:def_sigmaj}. Here,
$g_\infty(\lambda)=\bdensasymp[\psi]{0}{\lambda}{d}/\Kvar[\psi]{d}$, with $\bdensasymp[\psi]{0}{\cdot}{d}$ defined in \eqref{eq:bDpsi} and $\Kvar[\psi]{d}=\int_{-\infty}^{+\infty}|\xi|^{-2d}|\hat{\psi}(\xi)|^2\d\xi$
since by \cite[(26) and (29) in Theorem 1]{moulines:roueff:taqqu:2007:jtsa}
$$\left|\frac{\bdens[\phi,\psi]{j,0}{\lambda}{f}}{f^\ast(0)\Kvar[\psi]{d}2^{2dj}}
  - \frac{\bdensasymp[\psi]{0}{\lambda}{d}}{\Kvar[\psi]{d}}\right|
\leq C\,  L \,\Kvar[\psi]{d}^{-1}\, 2^{-\beta j}\to0\;,\text{ as } n\to\infty\;,$$
$$
\left|\frac{\sigma_j^2}{f^\ast(0)\Kvar[\psi]{d}2^{2dj}}-1\right|
\leq C\,  L \, 2^{-\beta j}\to0\;,\text{ as } n\to\infty\;.
$$
Note also that, by  \cite[Theorem 1]{moulines:roueff:taqqu:2007:jtsa},
$g_\infty(\lambda)$ is a continuous and $2\pi$-periodic
function on $(-\pi,\pi)$. Moreover, $g_\infty(\lambda)$ is bounded
on $(-\pi,\pi)$ by Lemma \ref{lem:D_inf:bounded} and
$$u_j=C_1 \frac{2^{-\beta j}}{\sigma_j^2/2^{2dj}}\left(2^{-\beta j}+C_2\frac{\sigma_j^2}{2^{2dj}}\right)\to 0,  \textrm{ as } n\to\infty\;,$$
 where $C_1$ and $C_2$ are positive constants.
The asymptotic expansion~\eqref{exp} for $\hat{\sigma}_{\MAD,j}$ can be deduced from the functional Delta method stated \textit{e.g} in \cite[Theorem~20.8]{vandervaart:1998} and the classical Delta Method stated \textit{e.g} in \cite[Theorem~3.1]{vandervaart:1998}. To show this, we have to prove that $T_0=T_1\circ T_2$ is Hadamard differentiable and that the corresponding Hadamard differential is defined and continuous on the whole space of cadlag functions. We prove first the Hadamard differentiability of the functional $T_1$. Let $(g_t)$ be a sequence of cadlag functions with bounded variations such that $\|g_t-g\|_\infty\to0$, as $t\to0,$ where $g$ is a cadlag function. For any non negative r, we consider
\begin{align*}
\frac{T_1(F_j+tg_t)[r]-T_1(F_j)[r]}{t}&=\frac{(F_j+tg_t)(r)-(F_j+tg_t)(-r)-F_j(r)+F_j(-r)}{t}\\
&=\frac{tg_t(r)-tg_t(-r)}{t}=g_t(r)-g_t(-r)\to g(r)-g(-r ),\quad
\end{align*}
since $\|g_t-g\|_\infty\to0$, as  $t\to0.$ The Hadamard differential of $T_1$ at $g$ is given by :
\begin{align*}
(DT_1(F_j).g)(r)=g(r)-g(-r).
\end{align*}
By \cite[Lemma~21.3]{vandervaart:1998}, $T_2$ is Hadamard differentiable. Finally, using the Chain rule \cite[Theorem~20.9]{vandervaart:1998}, we obtain the Hadamard differentiability of $T_0$ with the following Hadamard differential :
\begin{align*}
DT_0(F_j).g=-\frac{(DT_1(F_j).g)(T_0(F_j))}{(T_1(F_j))'[T_0(F_j)]}=-\frac{g(T_0(F_j))-g(-T_0(F_j))}{(T_1(F_j))'[T_0(F_j)]}.
\end{align*}
In view of the last expression, $DT_0(F_j)$ is a continuous function of $g$ and is defined on the whole space of cadlag functions. Thus by \cite[Theorem~20.8]{vandervaart:1998}, we obtain :
\begin{align*}
m(\Phi)\sqrt{n_j}\left(T_0(F_{n_j})-T_0(F_j)\right)=m(\Phi)DT_0(F_j)\left\{\sqrt{n_j}(F_{n_j}-F_j)\right\}+o_P(1),
\end{align*}
where $m(\Phi)$ is the constant defined in~\eqref{eq:mphi}.
Since $T_0(F_j)={\sigma_j}/{m(\Phi)}$ and
$(T_1(F_j))'(r)={2}{\sigma_j}^{-1}\varphi({r}/{\sigma_j}),$
where $\varphi$ is the p.d.f of a standard Gaussian random variable, we get
$$\sqrt{n_j}\left(\hat{\sigma}_{\MAD,j}-\sigma_j\right)=\frac{\sigma_j}{\sqrt{n_j}}
\sum_{i=0}^{n_j-1}\IF\left(\frac{W_{j,i}}{\sigma_j},\MAD,\Phi\right)+o_P(1)
$$ and the expansion~\eqref{exp} for
$\ast=\MAD$ follows from the classical Delta method applied with $f(x)=x^2$.
We end the proof of Proposition~\ref{lemma:asymp:exp} by proving the
asymptotic expansion \eqref{exp} for $\ast=\CR$.
We use the same arguments as those used previously. In this case the
Hadamard differentiability comes from
\cite[Lemma~1]{levy-leduc:boistard:2009}.
\end{proof}

The following theorem is an extension of \cite[Theorem~4]{arcones:1994} to arrays of stationary Gaussian processes in the multidimensional case.
\begin{theorem}\label{theo:ext:mult:arcones}
Let $\X_{J,i}=\big\{X^{(0)}_{J,i},\dots,X^{(d)}_{J,i}\big\}$ be
an array of standard stationary Gaussian processes such that for
$j,j'$ in $\{0,\dots,d\}$, the vector
$\big\{X_{J,i}^{(j)},X_{J,i}^{(j')}\big\}$
has a cross-spectral density $f_{J}^{(j,j')}$ and a
cross-correlation function $\rho_{J}^{(j,j')}$ defined by
$\rho_{J}^{(j,j')}(k)=\PE\big(X_{J,i}^{(j)}X_{J,i+k}^{(j')}\big)$,
for all $k\geq 0$.
Assume also that there exists a non increasing sequence $\{u_{J}\}_{J\geq1}$ such that $u_{J}$ tends to zero as $J$ tends to infinity and for all $J\geq1,$
\begin{equation}\label{eq:assumption:fJ0}
\sup_{\lambda\in(-\pi,\pi)}\big|f^{(j,j')}_{J}(\lambda)-g^{(j,j')}_\infty(\lambda)\big|\leq u_{J}\;,
\end{equation}
where $g^{(j,j')}_\infty$ is a $2\pi$-periodic function which is bounded on $(-\pi,\pi)$ and
continuous at the origin.
Let $h$ be a function on
$\Rset$ with Hermite rank $\tau\geq 1$ which is either bounded or is a
finite linear combination of Hermite polynomials.
Let $\boldsymbol{\beta}=\{\beta_0,\dots,\beta_d\}$ in $\Rset^{d+1}$ and $\mathcal{H}: \Rset^{d+1}\to \Rset$ the
real valued function defined by
$\mathcal{H}(\bx)=\sum_{j=0}^d\beta_jh(x_j)$. 
Let $\{n_J\}_{J\geq1}$ be a sequence of integers such that $n_J$ tends to infinity as $J$ tends to infinity. Then
\begin{equation}\label{eq:mult:arcones}
\frac{1}{\sqrt{n_{J}}}\sum_{i=1}^{n_{J}}\mathcal{H}\left(\X_{J,i}\right)\cd
\mathcal{N}\left(0,\tilde{\sigma}^2\right)\;,\textrm{ as } J\to\infty\; ,
\end{equation}
where
$$
\tilde{\sigma}^2=\lim_{n\to\infty}\PVar\Big(\frac{1}{\sqrt{n_{J}}}\sum_{i=1}^{n_{J}}
\mathcal{H}(\X_{J,i})\Big)=
\sum_{\ell\geq \tau}\frac{c_\ell^2}{\ell !}\sum_{0\leq j,j'\leq d} \beta_j\beta_{j'}(g^{(j,j')}_\infty)^{\star
  \ell}(0)\; .
$$
In the previous equality, $c_\ell=\PE[h(X)H_\ell(X)]$, where $H_\ell$ is the $\ell$-th Hermite
polynomial and $X$ is a standard Gaussian random variable.
\end{theorem}

The proof of Theorem \ref{theo:ext:mult:arcones} follows the same
lines as the one of Theorem \ref{theo:ext:arcones} and is thus omitted.

\begin{proof} [Proof of Theorem~\ref{theo:clt:wavelet}]

Without loss of generality, we set  $f^\ast(0)=1$.
In order to prove \eqref{eq:JointCentralLimitEmpVar}, let us first
prove that for
$\boldsymbol{\alpha}=(\alpha_{0},\dots,\alpha_\ell)$ where the
$\alpha_i$'s are in $\Rset$,
\begin{equation}\label{eq:clt3}
\sqrt{n2^{-J_0}}2^{-2J_0d}\sum_{j=0}^\ell\alpha_j\left(\hat{\sigma}^2_{\ast,J_0+j}(W_{J_0+j,0:n_{J_0+j}-1})-\sigma^2_{\ast,J_0+j}\right)\cd\mathcal{N}\left(0,\boldsymbol{\alpha}^T\mathbf{U}_\ast(d)\boldsymbol{\alpha}\right).
\end{equation}
By Proposition~\ref{lemma:asymp:exp},
\begin{multline}\label{eq:exp:mult}
\sqrt{n2^{-J_0}}2^{-2J_0d}\sum_{j=0}^\ell\alpha_j\left(\hat{\sigma}^2_{\ast,J_0+j}(W_{J_0+j,0:n_{J_0+j}-1})-\sigma^2_{\ast,J_0+j}\right)\\=\sum_{j=0}^\ell\frac{\sqrt{n2^{-J_0}}2^{-2J_0d}}{n_{J_0+j}}2\alpha_j\sigma^2_{J_0+j}\sum_{i=0}^{n_{J_0+j}-1}\IF\left(\frac{W_{J_0+j,i}}{\sigma_{J_0+j}},\ast,\Phi\right)+o_P(1).
\end{multline}
Thus, proving \eqref{eq:clt3} amounts to proving that
\begin{equation}\label{eq:expansion:wavelet}
\frac{2^{-\ell/2} f^{\ast}(0)\Kvar[\psi]{d}}{\sqrt{n_{J_0+\ell}}}  \sum_{j=0}^\ell
2\alpha_j
2^{2dj+j}
\sum_{i=0}^{n_{J_0+j}-1}\IF\left(\frac{W_{J_0+j,i}}{\sigma_{J_0+j}},\ast,\Phi\right)\cd\mathcal{N}\left(0,\boldmath{\alpha}^T\mathbf{U}_\ast(d)\boldmath{\alpha}\right)\;,
\end{equation}
since
$\sigma^2_{J_0+j}\sqrt{n2^{-J_0}}2^{-2J_0d}/n_{J_0+j}\sim
2^{2dj-\ell/2+j}\Kvar[\psi]{d}f^\ast(0)/\sqrt{n_{J_0+\ell}}$
, as $n$ tends to infinity,
by \cite[(29) in Theorem 1]{moulines:roueff:taqqu:2007:jtsa}.
Note that
\begin{multline*}
\sum_{i=0}^{n_{J_0+j}-1}\IF\left(\frac{W_{J_0+j,i}}{\sigma_{J_0+j}},\ast,\Phi\right)
=\sum_{i=0}^{n_{J_0+\ell}-1}\sum_{v=0}^{2^{\ell-j}-1}
\IF\left(\frac{W_{j+J_0,2^{\ell-j}i+v}}{\sigma_{J_0+j}},\ast,\Phi\right)\\
+\sum_{q=n_{J_0+j}-(\L-1)(2^{\ell-j}-1)}^{n_{J_0+j}-1}\IF\left(\frac{W_{j+J_0,q}}{\sigma_{J_0+j}},\ast,\Phi\right)
\end{multline*}
Using the notation: $\beta_j=2\alpha_j
2^{2dj-\ell/2+j}\Kvar[\psi]{d}f^\ast(0)$ and that $\IF$ is either bounded
or equal to $H_2/2$,
\begin{eqnarray*}
& &\frac{1}{\sqrt{n_{J_0+\ell}}}  \sum_{j=0}^\ell\beta_j
\sum_{i=0}^{n_{J_0+j}-1}\IF\left(\frac{W_{J_0+j,i}}{\sigma_{J_0+j}},\ast,\Phi\right)\\
&=&\frac{1}{\sqrt{n_{J_0+\ell}}}  \sum_{j=0}^\ell\beta_j
\sum_{i=0}^{n_{J_0+\ell}-1}\sum_{v=0}^{2^{\ell-j}-1}
\IF\left(\frac{W_{j+J_0,2^{\ell-j}i+v}}{\sigma_{J_0+j}},\ast,\Phi\right)+o_P(1)\\
&=&\frac{1}{\sqrt{n_{J_0+\ell}}}\sum_{i=0}^{n_{J_0+\ell}-1}
\F(Y_{J_0,\ell,i},\ast)+o_P(1)\; ,
\end{eqnarray*}
where
$$\F(Y_{J_0,\ell,i},\ast)=\sum_{j=0}^{\ell}\beta_j
\sum_{v=0}^{2^{\ell-j}-1}\IF\left(\frac{W_{j+J_0,2^{\ell-j}i+v}}{\sigma_{J_0+j}},\ast,\Phi\right) $$ and
\begin{multline*}
Y_{J_0,\ell,i}=\bigg(\frac{W_{J_0+\ell,i}}{\sigma_{J_0+\ell}},\frac{W_{J_0+\ell-1,2i}}{\sigma_{J_0+\ell-1}},\frac{W_{J_0+\ell-1,2i+1}}{\sigma_{J_0+\ell-1}},\dots,\frac{W_{J_0+j,2^{\ell-j}i}}{\sigma_{J_0+j}},\\
\dots\frac{W_{J_0+j,2^{\ell-j}i+2^{\ell-j}-1}}{\sigma_{J_0+j}},\dots,\frac{W_{J_0,2^{\ell}i}}{\sigma_{J_0}},\dots,\frac{W_{J_0,2^{\ell}i+2^\ell-1}}{\sigma_{J_0}}\bigg)^T
\end{multline*}
is a $2^{\ell+1}-1$ stationary Gaussian vector.
By Lemma \ref{lemma:rank}, $\F$ is of Hermite rank larger than 2.
Hence, from Theorem~\ref{theo:ext:mult:arcones} applied to
$\mathcal{H}(\cdot)=\F(\cdot),$ $\X_{J,i}=Y_{J_0,\ell,i}$ and $h(\cdot)=\IF(\cdot)$, we get
\begin{equation}
\label{eq:new:clt}
\frac{1}{\sqrt{n_{J_0+\ell}}}\sum_{i=0}^{n_{J_0+\ell}-1}\F(Y_{J_0,\ell,i},\ast)
\cd\mathcal{N}(0,\tilde{\sigma}^2_\ast)\;,
\end{equation}
where
$\tilde{\sigma}^2_\ast=
\lim_{n\to\infty}n_{J_0+\ell}^{-1}\PVar\Big(\sum_{i=0}^{n_{J_0+\ell}-1}\F(Y_{J_0,\ell,i},\ast)\Big)$.
By  \cite[(26) and (29)]{moulines:roueff:taqqu:2007:jtsa}
and by using the same arguments as those used in the proof of Proposition~\ref{lemma:asymp:exp},
Condition \eqref{eq:assumption:fJ0} of
Theorem~\ref{theo:ext:mult:arcones} holds with
$f_J^{(j,j')}(\lambda)=\bdenssingle[\psi]{^{(r)}}{J_0+j}{j-j'}{\lambda}{f}/{\sigma_{J_0+j}\sigma_{J_0+j'}}$
and
$g_\infty^{(j,j')}=
\bdensasympconv[\psi]{(r)}{j-j'}{\lambda}{d}/\Kvar{d}$, where $0\leq
r\leq 2^{j-j'}-1$
and $\bdens[\phi,\psi]{J_0+j,j-j'}{\cdot}{f}$ is the cross-spectral density of the stationary between scale
 process defined in (\ref{eq:def:cov:rob}).
Lemma \ref{lem:D_inf:bounded} and \cite[Theorem~1]{moulines:roueff:taqqu:2007:jtsa} ensure that
$\bdensasympconv[\psi]{(r)}{j-j'}{\cdot}{d}$ is a   bounded,
continuous and $2\pi$-periodic function.

By using Mehler's formula
\cite[Eq. (2.1)]{breuer:major:1983} and the expansion of $\IF$ onto the Hermite
polymials basis given by: $\IF(x,\ast,\Phi)=\sum_{p\geq 2}c_p(\IF_\ast) H_p(x)/p!$, where
$c_p(\IF_\ast)=\PE[\IF(X,\ast,\Phi)H_p(X)]$, $H_p$ being the $p$th
Hermite polynomial, we get 
\begin{multline}\label{eq:lim_variance}
\frac{1}{n_{J_0+\ell}}\PVar\Big(\sum_{i=0}^{n_{J_0+\ell}-1}\F(Y_{J_0,\ell,i},\ast)\Big)\\
=\frac{1}{n_{J_0+\ell}}\sum_{j,j'=1}^\ell\beta_j\beta_{j'}\sum_{i,i'=0}^{n_{J_0+\ell}-1}
\sum_{v=0}^{2^{\ell-j}-1}\sum_{v'=0}^{2^{\ell-j'}-1}
\PE\Big[\IF\Big(\frac{W_{J_0+j,2^{\ell-j}i+v}}{\sigma_{J_0+j}},\ast,\Phi\Big)\IF\Big(\frac{W_{J_0+j',2^{\ell-j'}i'+v'}}{\sigma_{J_0+j'}},\ast,\Phi\Big)\Big]\\
=\frac{1}{n_{J_0+\ell}}\sum_{j,j'=1}^\ell\beta_j\beta_{j'}\sum_{i=0}^{n_{J_0+j}-1}\sum_{i'=0}^{n_{J_0+j'}-1}
\PE\Big[\IF\Big(\frac{W_{J_0+j,i}}{\sigma_{J_0+j}},\ast,\Phi\Big)\IF\Big(\frac{W_{J_0+j',i'}}{\sigma_{J_0+j'}},\ast,\Phi\Big)\Big]+o(1)\\
= \frac{1}{n_{J_0+\ell}}\sum_{j,j'=1}^\ell \beta_j \beta_{j'}
\sum_{i=0}^{n_{J_0+j}-1}\sum_{i'=0}^{n_{J_0+j'}-1}
\sum_{p\geq 2} \frac{c_p^2(\IF_\ast)}{p!}
\PE\Big[\frac{W_{J_0+j,i}}{\sigma_{J_0+j}}\frac{W_{J_0+j',i'}}{\sigma_{J_0+j'}}\Big]^p+o(1)\; .
\end{multline}
Without loss of generality, we shall assume in the sequel that $j\geq j'$.
\eqref{eq:lim_variance} can be rewritten as follows by using that
$i'=2^{j-j'}q+r$, where $q\in\mathbb{N}$ and
$r\in\{0,1,\dots,2^{j-j'}-1\}$
and Eq. (18) in \cite{moulines:roueff:taqqu:2007:jtsa}
\begin{multline*}
 \frac{1}{n_{J_0+\ell}}\sum_{j,j'=1}^\ell \beta_j\beta_{j'}
\sum_{i=0}^{n_{J_0+j}-1}\sum_{q=0}^{n_{J_0+j}-1}\sum_{r=0}^{2^{j-j'}-1}
\sum_{p\geq 2} \frac{c_p^2(\IF_\ast)}{p!}
\PE\Big[\frac{W_{J_0+j,0}}{\sigma_{J_0+j}}\frac{W_{J_0+j',2^{j-j'}(q-i)+r}}{\sigma_{J_0+j'}}\Big]^p+o(1)\\
=\frac{  n_{J_0+j}}{n_{J_0+\ell}}\sum_{j,j'=1}^\ell \beta_j\beta_{j'}
\sum_{|\tau|<n_{J_0+j}}\sum_{r=0}^{2^{j-j'}-1}
\sum_{p\geq 2} \frac{c_p^2(\IF_\ast)}{p!}\Big(1-\frac{|\tau|}{n_{J_0+j}}\Big)\PE\Big[\frac{W_{J_0+j,0}}{\sigma_{J_0+j}}\frac{W_{J_0+j',2^{j-j'}\tau+r}}{\sigma_{J_0+j'}}\Big]^p+o(1)\\
=\frac{ n_{J_0+j}}{n_{J_0+\ell}}\sum_{j,j'=1}^\ell \beta_j\beta_{j'}
\sum_{|\tau|<n_{J_0+j}}\sum_{r=0}^{2^{j-j'}-1}
\sum_{p\geq  2}\frac{c_p^2(\IF_\ast)}{p!}\Big(1-\frac{|\tau|}{n_{J_0+j}}\Big)
\Big(\int_{-\pi}^{\pi}\frac{\bdenssingle[\psi]{^{(r)}}{J_0+j}{j-j'}{\lambda}{f}\rme^{\rmi\lambda\tau}}
{\sigma_{J_0+j}\sigma_{J_0+j'}}\d\lambda\Big)^p  +o(1)\;,
\end{multline*}
where $\bdens[\phi,\psi]{J_0+j,j-j'}{\cdot}{f}$ is the cross-spectral density of the stationary between scale
 process defined in (\ref{eq:def:cov:rob}).
We aim at applying Lemma~\ref{lemma:double:fatou} with $f_n$, $g_n$,
$f$ and $g$ defined hereafter.
$$f_{n_{J_0+j}}(\tau,p)=\frac{c_p^2(\IF_\ast)}{p!}\sum_{r=0}^{2^{j-j'}-1}
\1\{|\tau|<n_{J_0+j}\}
\Big(1-\frac{|\tau|}{n_{J_0+j}}\Big)\PE\Big[\frac{W_{J_0+j,0}}{\sigma_{J_0+j}}\frac{W_{J_0+j',2^{j-j'}\tau+r}}{\sigma_{J_0+j'}}\Big]^p \;.
$$
Observe that $|f_{n_{J_0+j}}|\leq g_{n_{J_0+j}}$, where
$$
g_{n_{J_0+j}}(\tau,p)=\frac{c_p^2(\IF_\ast)}{p!}
\sum_{r=0}^{2^{j-j'}-1}\1\{|\tau|<n_{J_0+j}\}
\Big(1-\frac{|\tau|}{n_{J_0+j}}\Big)
\PE\Big[\frac{W_{J_0+j,0}}{\sigma_{J_0+j}}\frac{W_{J_0+j',2^{j-j'}\tau+r}}{\sigma_{J_0+j'}}\Big]^2\;.
$$
Using \cite[(26) and (29) in Theorem
1]{moulines:roueff:taqqu:2007:jtsa} we get that
$$\lim_{n\to\infty}\frac{\bdens[\phi,\psi]{J_0+j,j-j'}{\lambda}{f}}{\sigma_{J_0+j}\sigma_{J_0+j'}}
=\frac{2^{d(j-j')}}{\Kvar[\psi]{d}}\bdensasymp[\psi]{j-j'}{\lambda}{d}\;.$$
This implies that $\lim_{n\to\infty}f_{n_{J_0+j}}(\tau,p)=f(\tau,p)$ where
$$
f(\tau,p)=\frac{c_p^2(\IF_\ast)}{p!}\sum_{r=0}^{2^{j-j'}-1}
\Big(\frac{2^{d(j-j')}}{\Kvar[\psi]{d}}\int_{-\pi}^{\pi}\bdensasympconv[\psi]{(r)}{j-j'}{\lambda}{d}\rme^{\rmi\lambda\tau}
\d\lambda\Big)^p\;.
$$
Futhermore, $\lim_{n\to\infty}g_{n_{J_0+j}}(\tau,p)=g(\tau,p)$ where
$$
g(\tau,p)=\frac{c_p^2(\IF_\ast)}{p!}\frac{2^{2d(j-j')}}{\Kvar[\psi]{d}^2}\Big|\int_{-\pi}^{\pi}\bdensasymp[\psi]{j-j'}{\lambda}{d}\rme^{\rmi\lambda\tau}
\d\lambda\Big|^2_2\;,
$$
and $|\bx|_2^2=\sum_{k=1}^r x_k^2$ for $\bx=(x_1,\dots,x_r)\in\mathbb{R}^r.$
Using (63)-(65) in \cite{moulines:roueff:taqqu:2007:jtsa} we get
$$
\sum_{p\geq 2}\sum_{\tau\in\mathbb{Z}}g_{n_{J_0+j}}(\tau,p)
\longrightarrow\Big(\sum_{p\geq
  2}\frac{c_p^2(\IF_\ast)}{p!}\Big)\frac{2^{2d(j-j')}}{\Kvar[\psi]{d}^2}2\pi
\int_{-\pi}^{\pi}|\bdensasymp[\psi]{j-j'}{\lambda}{d}|_2^2
\d\lambda\;,\textrm{ as }n\to\infty\;,
$$
Then, with Lemma~\ref{lemma:double:fatou}, we obtain
\begin{multline*}
\tilde{\sigma}^2_\ast=\sum_{p\geq2}\frac{c_p^2(\IF_\ast)(f^\ast(0))^2}{p!\Kvar[\psi]{d}^{p-2}}
\sum_{j,j'=0}^\ell
4\alpha_j\alpha_{j'}2^{dj(2+p)}2^{dj'(2-p)+j'}\sum_{\tau\in\Zset}\sum_{r=0}^{2^{j-j'}-1
}\Big(\int_{-\pi}^{\pi}\bdensasympconv[\psi]{(r)}{j-j'}{\lambda}{d}\rme^{\rmi\lambda\tau}\d\lambda\Big)^p\;.
\end{multline*}
\end{proof}


\section{Technical Lemmas}\label{sec:lemmas}

\begin{lemma}
\label{lemma:rank}
Let $X$ be a standard Gaussian random variable. The influence
functions $\IF$ defined in Proposition~\ref{lemma:asymp:exp} have the following properties
\begin{align}\label{eq1:cl}
\PE[\IF(X,\ast,\Phi)]=0\;,
\\\label{eq2:cl}
\PE[X\IF(X,\ast,\Phi)]=0\;,\\\label{eq3:cl}
\PE[X^2\IF(X,\ast,\Phi)]\neq0\;.
\end{align}
\end{lemma}
\begin{proof} [Proof of Lemma~\ref{lemma:rank}]
We only have to prove the result for $\ast=\MAD$ since the result
for $\ast=\CR$ follows from \cite[Lemma~12]{levy-leduc:boistard:2009}.
\eqref{eq1:cl} comes from $\PE(\1_{\{X\leq
  1/m(\Phi)\}})=\PE(\1_{\{X\leq\Phi^{-1}(3/4)\}})=3/4$ and
$\PE(\1_{\{X\leq  -1/m(\Phi)\}})=1/4$, where $X$ is a
standard Gaussian random variable. \eqref{eq2:cl} follows from
$\int_\Rset x\1_{\{x\leq \Phi^{-1}(3/4)\}}\varphi(x)\d x-\int_\Rset
x\1_{\{x\leq -\Phi^{-1}(3/4)\}}\varphi(x)\d x
=-\varphi(\Phi^{-1}(3/4))+\varphi(-\Phi^{-1}(3/4))=0$, where $\varphi$ is the
p.d.f. of a standard Gaussian random variable and the fact that $\PE(X)=0$.
Let us now compute $\PE[X^2\IF(X,\MAD,\Phi)].$
Integrating by parts, we get
$\int_\Rset x^2\1_{\{x\leq \Phi^{-1}(3/4)\}}\varphi(x)\d x-3/4
-\int_\Rset x^2\1_{\{x\leq -\Phi^{-1}(3/4)\}}\varphi(x)\d x+1/4=-2\varphi\left(\Phi^{-1}(3/4)\right).$
Thus, $\PE[X^2\IF(X,\MAD,\Phi)]=2\neq0$,
which concludes the proof.
\end{proof}

\begin{lemma}\label{lemma:if}
Let $X$ be a standard Gaussian random variable. The influence
functions $\IF$ defined in Lemma \ref{lemma:asymp:exp}
have the following properties
\begin{align}
\label{eq2}
\PE[\IF^2(X,\MAD,\Phi)]&=\frac{m^2(\Phi)}{16\varphi\left(\Phi^{-1}(3/4)^2\right)}=1.3601\;,\\\label{eq3}
\PE[\IF^2(X,\CR,\Phi)]&\approx 0.6077\;.
\end{align}
\end{lemma}

\begin{proof} [Proof of Lemma~\ref{lemma:if}]
Eq \eqref{eq3} comes from \cite{rousseeuw:croux:1993}.
Since ,
$$\PE[\IF^2(X,\MAD,\Phi)]=\frac{m^2(\Phi)}{4\varphi\left(\Phi^{-1}(3/4)^2\right)}
\PVar\left(\1_{\{|X|\leq\Phi^{-1}(3/4)\}}\right)\;,$$
where $\1_{\{|X|\leq\Phi^{-1}(3/4)\}}$ is a Bernoulli random variable
with parameter 1/2, \eqref{eq2} follows.
\end{proof}

\begin{lemma}\label{lem:conv:integrale}
$$
\int_{\mathbb{R}^2}\frac{1}{1+|\mu-\mu'|}\frac{1}{1+|\mu|}\frac{1}{1+|\mu'|}\d\mu\d\mu'<\infty\;.
$$
\end{lemma}
\begin{proof} [Proof of Lemma~\ref{lem:conv:integrale}]
Let us set
$I=\int_{-\infty}^{\infty}\int_{-\infty}^{\infty}p(\mu,\mu')\d\mu\d\mu',$
with
$$p(\mu,\mu')=\frac{1}{1+|\mu-\mu'|}\frac{1}{1+|\mu|}\frac{1}{1+|\mu'|}\;.$$
Note that $I=I_1+I_2+I_3+I_4,$ where
$I_1=\int_{0}^{\infty}\int_{0}^{\infty}p(\mu,\mu')\d\mu\d\mu',$
$I_2=\int_{0}^{\infty}\int_{-\infty}^{0}p(\mu,\mu')\d\mu\d\mu',$
$I_3=\int_{-\infty}^{0}\int_{0}^{\infty}p(\mu,\mu')\d\mu\d\mu'$ and
$I_4= \int_{-\infty}^{0}\int_{-\infty}^{0}p(\mu,\mu')\d\mu\d\mu'.$ It
is easy to see that $I_1=I_4$ and $I_2=I_3$. Let us now compute $I_1$. Using partial fraction decomposition,
\begin{multline*}
I_1=\int_0^\infty\frac{1}{1+\mu'}\Big(\int_{\mu'}^{\infty}\frac{1}{1+\mu-\mu'}\frac{1}{1+\mu}\d\mu\Big)\d\mu'
+\int_0^\infty\frac{1}{1+\mu'}\Big(\int_{0}^{\mu'}\frac{1}{1-\mu+\mu'}\frac{1}{1+\mu}\d\mu\Big)\d\mu'\,\\
=\int_0^\infty \frac{\log(1+\mu')}{\mu'(1+\mu')}\d\mu'+2\int_0^\infty \frac{\log(1+\mu')}{(2+\mu')(1+\mu')}\d\mu'<\infty\;,
\end{multline*}
since in the neighborhood of 0, $\log(1+\mu')/\{\mu'(1+\mu')\}\sim
1/(1+\mu')$,
$\log(1+\mu')/\{(2+\mu')(1+\mu')\}\sim \{-1/(1+\mu')+2/(2+\mu')\}$
and in the neighborhood of $\infty$, $\log(1+\mu')/\{\mu'(1+\mu')\}$
and $\log(1+\mu')/\{(2+\mu')(1+\mu')\}\sim\log(\mu')/\mu'^{2}.$
Let us now compute $I_2.$ Using the same arguments as previously, we get
$$
I_2=\int_0^\infty\frac{1}{1+\mu'}\Big(\int_{0}^{\infty}\frac{1}{1+\mu+\mu'}\frac{1}{1+\mu}\d\mu\Big)\d\mu'=\int_0^\infty \frac{\log(1+\mu')}{\mu'(1+\mu')}\d\mu'<\infty\;.
$$
\end{proof}

\begin{lemma}\label{lem:conv:integrale:even}
$$
\int_{\mathbb{R}}\bigg(\int_{\mathbb{R}}\frac{1}{1+|\mu+\mu'|}\frac{1}{1+|\mu|}\d\mu\bigg)^2\d\mu'<\infty\;.
$$
\end{lemma}
\begin{proof} [Proof of Lemma~\ref{lem:conv:integrale:even}]
Let us set
$I=\int_{-\infty}^{\infty}\big(\int_{-\infty}^{\infty}p(\mu,\mu')\d\mu\big)^2\d\mu',$
with
$$p(\mu,\mu')=\frac{1}{1+|\mu+\mu'|}\frac{1}{1+|\mu|}\;.$$
Note that $I\leq 2(I_1+I_2+I_3+I_4),$ where
$I_1=\int_{0}^{\infty}\big(\int_{0}^{\infty}p(\mu,\mu')\d\mu\big)^2\d\mu',$
$I_2=\int_{0}^{\infty}\big(\int_{-\infty}^{0}p(\mu,\mu')\d\mu\big)^2\d\mu',$
$I_3=\int_{-\infty}^{0}\big(\int_{0}^{\infty}p(\mu,\mu')\d\mu\big)^2\d\mu'$ and
$I_4= \int_{-\infty}^{0}\big(\int_{-\infty}^{0}p(\mu,\mu')\d\mu\big)^2\d\mu'.$ It
is easy to see that $I_1=I_4$ and $I_2=I_3$. Let us now compute $I_1$. Using partial fraction decomposition,
$$
I_1=\int_0^\infty\Big(\int_{0}^{\infty}\frac{1}{1+\mu+\mu'}\frac{1}{1+\mu}\d\mu\Big)^2\d\mu'
=\int_0^\infty\Big(\frac{1}{\mu'}\log(1+\mu')\Big)^2\d\mu'<\infty
$$
since in the neighborhood of 0, $[\log(1+\mu')]^2/{\mu'}^2\sim
1$,
and in the neighborhood of $\infty$, $[\log(1+\mu')]^2/{\mu'}^2 \sim[\log(\mu')]^2/\mu'^{2}.$
Let us now compute $I_2.$ Using the same arguments as previously, we
get that there exists a positive constant $C$ such that
\begin{multline*}
I_2\leq 2\int_0^\infty\Big(\int_{\mu'}^{\infty}\frac{1}{1+\mu-\mu'}\frac{1}{1+\mu}\d\mu\Big)^2\d\mu'
+2\int_0^\infty\Big(\int_{0}^{\mu'}\frac{1}{1-\mu+\mu'}\frac{1}{1+\mu}\d\mu\Big)^2\d\mu'\,\\
\leq C\int_0^\infty \bigg[\frac{\log(1+\mu')}{\mu'}\bigg]^2\d\mu'+C\int_0^\infty \bigg[\frac{\log(1+\mu')}{(2+\mu')}\bigg]^2\d\mu'<\infty\;.
\end{multline*}
\end{proof}

\begin{lemma}\label{lem:D_inf:bounded}
Let
$
\be_{u}(\xi) = 2^{-{u}/2}\, [1, \rme^{-\rmi2^{-u}\xi}, \dots,
\rme^{-\rmi(2^{u}-1)2^{-u}\xi}]^T,
$
where $\xi\in\Rset.$ For all $u\geq 0$, each component of the vector
\begin{equation*}
\bdensasymp[\psi]{u}{\lambda}{d} =
\sum_{l\in\Zset} |\lambda+2l\pi|^{-2d}\,\be_{u}(\lambda+2l\pi) \,
\overline{\hat{\psi}(\lambda+2l\pi)}\hat{\psi}(2^{-u}(\lambda+2l\pi))\;,
\end{equation*}
is bounded on $(-\pi,\pi)$, where $\hat{\psi}$ is defined in \eqref{eq:fourier_psi}.
\end{lemma}
\begin{proof}[Proof of Lemma \ref{lem:D_inf:bounded}] We start with
  the case where $l=0.$ Using (\ref{eq:MVM:rob}), we obtain that\\
$2^{-u/2}|\lambda|^{-2d}|\hat{\psi}(\lambda)||\hat{\psi}(2^{-u}\lambda)|=O(|\lambda|^{2M-2d})$, as
$\lambda\to0$ hence,
\eqref{eq:ConditionD} ensures that \\
$2^{-u/2}|\lambda|^{-2d}|\hat{\psi}(\lambda)||\hat{\psi}(2^{-u}\lambda)|=O(1).$
Let $\be_{u}^{(k)}$ denotes the $k$-th component of the vector $\be_{u}$.
For $l\neq0$, \ref{item:psiHat} ensures that for all $\lambda$ in
$(-\pi,\pi)$ there exists a positive constant $C$ such that
$|\hat{\psi}(\lambda)|\leq {C}/{(1+|\lambda|)^\alpha}$. Then, there
exists a positive constant $C'$ such that
$$
\sum\limits_{l\in\Zset^\ast}|\lambda+2\pi
l|^{-2d}\overline{\hat{\psi}(\lambda+2\pi l)}\hat{\psi}\big(2^{-u}(\lambda+2\pi l)\big)\be_{u}^{(k)}(\lambda)
\leq C' \sum_{l\in\Zset^*}|\lambda+2\pi l|^{-2d-2\alpha}\; .
$$
If $\lambda=0$, $\sum\limits_{l\in\Zset^*}{1}/{|2\pi
  l|^{2d+2\alpha}}<\infty$ by \eqref{eq:ConditionD}.
If $\lambda\neq0$, then, since $-\pi\leq \lambda\leq \pi$,
$\sum\limits_{l\in\Zset^*}{1}/{|\lambda+2\pi
  l|^{2d+2\alpha}}\leq\sum\limits_{l\in\Zset^*}{1}/{|\pi(
  2l-1)|^{2d+2\alpha}}<\infty $ by \eqref{eq:ConditionD}.
\end{proof}
\begin{lemma}\label{lemma:double:fatou}
Let $f_n$ and $g_n$ be two sequences of measurable functions on a
measure space $(\Omega,\mathcal{F},\mu)$
such that for all $n$ $|f_n|\leq g_n$.
Assume that $\underset{n\to\infty}{\lim\inf} g_n$ exists and is equal
to $g$. Assume also that $\int g\d\mu= \underset{n\to\infty}{\lim\inf}\int g_n\d\mu$
and $\lim\limits_{n\to\infty}f_n=f.$
Then $\int\underset{n\to\infty}{\lim\inf} f_n\d\mu= \underset{n\to\infty}{\lim\inf}\int f_n\d\mu$.
\end{lemma}
\begin{proof}[Proof of Lemma~\ref{lemma:double:fatou}]
Since $f_n=f_n^+-f_n^-$, where $f_n^+,f_n^-\geq 0$, we assume in the
sequel that $f_n$ is non negative. By Fatou's Lemma
$\int \underset{n\to\infty}{\lim\inf}(g_n-f_n)\d\mu \leq
\underset{n\to\infty}{\lim\inf}\int (g_n-f_n)\d\mu $. Using that
$\underset{n\to\infty}{\lim\inf}g_n=g$ and that
$\int g\d\mu= \underset{n\to\infty}{\lim\inf}\int g_n\d\mu$, we
obtain $\underset{n\to\infty}{\lim\sup}\int f_n\d\mu\leq
\int\underset{n\to\infty}{\lim\sup} f_n\d\mu.$
By applying Fatou's Lemma to $f_n$, we obtain
$\int\underset{n\to\infty}{\lim\inf}
f_n\d\mu\leq\underset{n\to\infty}{\lim\inf} \int f_n\d\mu$.
Thus,
$$\int f\d\mu=\int\underset{n\to\infty}{\lim\inf}
f_n\d\mu\leq\underset{n\to\infty}{\lim\inf} \int f_n\d\mu
\leq\underset{n\to\infty}{\lim\sup}\int f_n\d\mu\leq
\int\underset{n\to\infty}{\lim\sup} f_n\d\mu=\int f\d\mu\;,$$
which concludes the proof.
\end{proof}

\bibliography{lrd}
\bibliographystyle{abbrv}
\end{document}